\documentclass[10pt]{amsart}
\usepackage{amsmath,amsfonts,amssymb,amsthm,amscd,latexsym,euscript,xypic, verbatim}
\usepackage{epsf}
\usepackage{epsfig}
\usepackage{amstext}
\usepackage{xspace}
\usepackage{amsfonts}
\usepackage{amsmath}
\usepackage{amssymb}
\usepackage{xspace}
\usepackage{graphicx}
\usepackage[active]{srcltx}
\usepackage[all]{xy}
\usepackage{graphics}
\usepackage{color}

\usepackage{tikz}
\usepackage{caption}
\usepackage[hidelinks]{hyperref}
\usetikzlibrary{patterns}
\usepackage{fp}
\newcommand\Vect{\operatorname{Vect}}
\colorlet{agreen}{green!50!black}
\newcommand\ZZ{\mathbf{Z}}
\newcommand\NN{\mathbf{N}}

\swapnumbers
\theoremstyle{plain}

\newtheorem*{thm1.2}{(1.2) Theorem}
\newtheorem*{thm1.3}{(1.3) Theorem}
\newtheorem*{thm1.4}{(1.4) Theorem}
\newtheorem*{propA*}{Proposition A}
\newtheorem*{propB*}{Proposition B}
\newtheorem*{thmC*}{Theorem C}
\newtheorem*{propD*}{Proposition D}
\newtheorem{prop}{Proposition}[section]

\newtheorem{cor}[prop]{Corollary}

\theoremstyle{definition}
\newtheorem{point}[prop]{}

\newtheorem{Def}[prop]{Definition}
\newtheorem*{Def*}{Definition}

\newtheorem{example}[prop]{Example}

\newtheorem*{notation*}{Notation}
\newtheorem*{question*}{Question}

\newcommand{\caln}{\mathcal N}

\newcommand{\calx}{\mathcal X}
\newcommand{\calc}{\mathcal C}
\newcommand{\RR}{{\mathbf R}}
\newcommand{\QQ}{{\mathbf Q}}

\newcommand{\N}{\mathbf N}

\newcommand{\V}{\mathcal V}
\newcommand{\cS}{\mathcal S}
\newcommand{\cT}{\mathcal T}
\newcommand{\W}{\mathcal W}

\newcommand{\ra}{\rightarrow}

\xyoption{2cell}

\title{Multidimensional persistence and noise}

\author[M. Scolamiero]{Martina Scolamiero}
\address{Martina Scolamiero, EPFL SV BMI UPHESS\\  MA B3 495 (Batiment MA) Station 8\\ CH-1015 Lausanne\\
Switzerland}
\email{martina.scolamiero@epfl.ch}

\author[W. Chach\'olski]{Wojciech Chach\'olski}
\address{Corresponding author: Wojciech Chach\'olski\\ Department of Mathematics\\ KTH Stockholm
Lindstedtsv\" agen 25 \\ 10044 Stockholm \\ Sweden}
\email{wojtek@math.kth.se}
\thanks{W. Chach\'olski is a corresponding author}
\thanks{W. Chach\'olski was partially supported by G\"oran Gustafsson Stiftelse and VR grants.}

\author[A. Lundman]{Anders Lundman}
\address{Anders Lundman\\ Department of Mathematics\\ KTH Stockholm
Lindstedtsv\" agen 25 \\ 10044 Stockholm \\ Sweden}
\email{alundman@kth.se}

\author[R. Ramanujam]{Ryan Ramanujam}
\address{Ryan Ramanujam, Department of Mathematics\\ KTH Stockholm
Lindstedtsv\" agen 25 \\ 10044 Stockholm \\ Sweden}
\email{Ryan.Ramanujam@ki.se}

\author[S. \"Oberg]{Sebastian \"Oberg}
\address{Sebastian \"Oberg, Department of Mathematics\\ KTH Stockholm
Lindstedtsv\" agen 25 \\ 10044 Stockholm \\ Sweden}
\email{sobe@kth.se}

\subjclass[2010]{Primary: 55, 18, 68}
\thanks{Communicated by Peter Olver}
\begin{document}

\maketitle 

\begin{abstract}
In this paper we study multidimensional persistence modules~\cite{Multi,Interleaving Multi}
 via what we call tame functors and noise systems.  A noise system leads to a
pseudo-metric topology on the category of tame functors. We show how this  pseudo-metric can be used to identify persistent features
of compact multidimensional persistence modules. To count such features we introduce the feature counting  invariant
and prove that assigning this invariant to compact tame functors is  a  $1$-Lipschitz operation.
For $1$-dimensional persistence, we explain how, by choosing an appropriate noise system, the feature counting invariant  identifies  the same persistent features as the classical barcode construction.
\end{abstract}

\section{Introduction}
The aim of this paper is to
 present a new  perspective on  \emph{multidimensional persistence} \cite{Multi} and  introduce a tool for creating numerous new invariants for  multidimensional persistence modules. This new tool
helps in extracting  information   by purposely  defining  what is not wanted.
We do that by introducing the concept of a noise system  and show how it leads to a continuous invariant.
For one-dimensional persistence~\cite{persistence} and an appropriate choice of a noise system
this invariant turns out to be closely related to  the well-studied  \emph{barcode}.
  The barcode in one-dimensional persistence has proven itself to be  a valuable tool for analysing data from a variety of   different research areas (see e.g. \cite{NaturalImages}, \cite{ViralEvolve}, \cite{2DRobots} and \cite{SignalAnalysis}).
 Multidimensional persistence however has not yet had as much  use in data analysis, even though its  potential is  even exceeding that of one-dimensional persistence.
  As an example whenever one has multiple measurements and wants to understand the relations between them this naturally translates  into a multidimensional persistent module. 
Furthermore when one studies a space using a sampling, i.e. point cloud data, multidimensional persistence  can provide additional insight into the geometrical  properties of the space.   
These types of applications can   be found for example in \cite{HepaticLesions}, where hepatic
 lesions were classified using multidimensional persistence, or in \cite{ImageRetrival}, where multidimensional persistence was used to help with content-based image retrieval. 

The pipeline for using multidimensional persistence  for data analysis starts with a  choice of multiple measurements on a data set.   These measurements are   used   (via for example the \v{C}ech or Vietoris-Rips constructions, see~\cite{Carlsson review}) to  get   a topological space or a sequence of such, resulting  in a  functor $X\colon  \QQ^r\to \text{Spaces}$ where $\QQ^r$ is the poset of $r$-tuples of non-negative rational numbers (see~\ref{point NRposets}) (we use rational  instead of real numbers in order to avoid certain technical difficulties). 
The aim is to gain new insight into the data by extracting homological  information out of these spaces.
Applying the $i$-th homology with coefficients in a field $K$ gives us  a functor $H_i(X,K)\colon  \QQ^r\to  \text{Vect}_K$  called an $r$-dimensional persistence module.  The functors obtained in this way are often  \emph{tame}
(see Definition~\ref{def tame}).  The  category  of tame functors $\text{\rm Tame}(\QQ^r,\text{\rm Vect}_K)$ has very similar properties to   the category of graded modules over the $r$-graded polynomial ring $K[x_1,\ldots,x_r]$. In the case  $r=1$  this translates into the barcode being  a complete discrete invariant for one-dimensional compact and tame persistence modules. For $r>1$, it is known that no such discrete invariant can exist, as in this case the moduli space of $r$-dimensional compact and tame persistence modules is a positive dimensional algebraic variety, (see~\cite{Multi}).  Furthermore, this variety is complicated enough that 
there is  simply no realistic hope  to  find easily visualisable and continuous  invariants completely describing their compact objects.
 In our opinion  however  looking for complete invariants of algebraic objects such as the multidimensional persistence modules is  not the main goal of topological data analysis.
For data analysis it is much more useful to be able to extract out of such modules their {\em continuous}  features.
That is why we propose that  instead of focusing on the objects in $\text{\rm Tame}(\QQ^r,\text{\rm Vect}_K)$ we study relations between them using topology  and metrics as  main tools.

By defining collections of tame functors that are $\epsilon$-small, for every non-negative rational number $\epsilon$, we create what we call a noise system in Definition~\ref{def noise}. These noise systems help us to tell the size of tame functors and thus also which of these functors we can disregard and consider as noise. This leads us to define a pseudo-metric on the category $\text{\rm Tame}(\QQ^r,\text{\rm Vect}_K)$ in Definition \ref{def distance tame} and hence also induce a topology on that category. Equipped with this topology we define invariants called feature counting functions in Section 9. These invariants are functions $\text{bar}(F)\colon\QQ \to \NN$ (with values in the set of natural numbers) which for a given functor $F$ and a positive rational number $\epsilon$ return the smallest rank of a functor in an $\epsilon$-neighbourhood of $F$. We then show in Proposition~\ref{prop:lipschitz} that the assignment $F\mapsto \text{bar}(F)$ is  not just continuous, but also 1-Lipschitz with respect to the topology we just introduced. A standard  way   of producing  invariants of multidimensional persistence modules   is a reduction to one-dimensional case
by restricting the modules   to one parameter submodules and then using one persistence. This is the key idea behind  invariants such as the rank invariant~\cite{Multi} and more generally multidimensional PBNs in~\cite{PBN}. By an appropriate choice of a noise system, these can be recovered as  feature counting functions. For arbitrary noises however, feature counting functions  provide a much wider set of stable invariants for multidimensional persistence modules, invariants that go beyond
the  reduction  to the one-dimensional case.


\medskip

\paragraph{\em Organisation of the Paper.} In Section $2$ we go through the notation and background needed for the paper. 
This continues in Section $3$ which contains  some background on functors indexed by $r$-tuples of the natural numbers. These results will be crucial in Section $4$ where we instead look at functors indexed by $r$-tuples of  the rational numbers and introduce the concept of tameness for functors. 
Although related, our notion of tameness for functors  is not exactly analogous to
the concept of tame functions described  in~\cite{stable_persistence}.
In Section $5$  we prove some fundamental  properties and show how to compute certain homological  invariants of tame and compact functors. 
Tame and compact functors are  our main object of study. Such functors are less general than $q$-tame persistence modules as defined in \cite{qtame} since compact functors have finite dimensional values. Nevertheless many applications, as for example the ones defined in \cite{Multi}, can be modelled through such objects.

Section $6$ contains the definition of a noise system and  several examples of different explicit noise systems. We explore this further in Section $7$ where we look at under which circumstances a noise system is closed under direct sums. Section $8$ then uses the notion of a noise system to define a pseudo-metric on tame functors inducing  a topology on such functors. This allows us to define  noise dependent invariants, called feature counting invariants in Section $9$, which we prove are 1-Lipschitz  with respect to the  pseudo-metric. In Section 10 we show that the feature counting invariants generalise the barcode from one-dimensional persistence. We provide a simplified description of the feature counting invariant for the standard noise  in the multidimensional case in Section $11$. This noise system is the most natural one with respect to what is typically considered to be noise in multipersistence via the interleaving distance of~\cite{Interleaving Multi}. In Section 12 we describe  the notion of denoising and how it (hopefully) can help with   
computing the feature counting invariants.  In section 13 we outline some possible directions for future results. Lastly in the Appendix (Section 14) we prove, for completeness, properties of the category of vector space valued functors indexed by $\N^r$ and construct minimal covers in this category.
\medskip

 We would like to thank Claudia Landi for inspiring discussions about stability.
 
\section{Notation and Background}
\begin{point}\label{point categories}
The symbols $\text{Sets}$ and  $\text{Vect}_K$  denote the categories of respectively sets and   $K$-vector spaces
where  $K$ is always assumed to be a field.  Given a $K$-vector space,  we denote its dimension by $\text{dim}_K V$.  The linear span functor, denoted by $K:\text{Sets}\ra \text{Vect}_K$, 
 assigns to a set $S$ the vector space $K(S):=\oplus_S K$ with base $S$ and to a function $f\colon S\to S'$ 
  the  homomorphism $K(f)\colon\oplus_S K\to \oplus_{S'} K$
 given by $f$ on the bases.
 \end{point}

 \begin{point}\label{point functors}
 Let $I$ be a small category and $\calc$ be a category. The symbol $\calc^{\text{op}}$ denotes the opposite category to $\calc$ and $\text{Fun}(I,\calc)$ denotes the category of functors indexed by $I$ with values  in  $\calc$ and natural transformations as morphisms (see \cite{Categories}).
We use the symbol
$\text{Nat}(F,G)$   to denote the set of natural transformations between two functors 
$F,G\colon I\to \calc$. If $\calc$ is  abelian, then so is  $\text{Fun}(I,\calc)$. A sequence of composable morphisms in  $\text{Fun}(I,\calc)$
is exact if and only if its values  at any object $i$ in $I$  form an exact sequence in $\calc$.
If $\calc$ has enough projective objects, then so does $\text{Fun}(I,\calc)$.
\end{point}

\begin{point}\label{pt multiset}
Let $X$ be a set. A {\bf multiset} on $X$ is a function $\beta\colon X\to \N$ of sets where $\N$ denotes the set of natural numbers.
A  multiset  $\beta$ is {\bf finite} if  $\beta(x)\not=0$ for only finitely many $x$ in $X$.  If $\beta$ is a finite multiset on $X$, then its {\bf size}  is defined as 
$\sum_{x\in X}\beta(x)$.
We say that $\beta\colon X\to \N$ is a {\bf subset}  of $\gamma\colon X\to \N$ if $\beta(x)\leq \gamma(x)$ for any $x$ in $X$.
\end{point}

  \begin{point}\label{point freefunctors}
%
Let $i$ be an object in a small category $I$. The symbol $K_{I}(i,-)\colon  I\to \text{Vect}_K$ denotes  the composition of the representable functor $\text{mor}_{I}(i,-)\colon I\to \text{Sets}$ with the
 linear span functor  $K\colon\text{Sets}\to \text{Vect}_K$. This functor is called {\bf free on one generator}. 
 We often omit the subscript $I$ and write $K(i,-)$.

Let  $\{V_i\}_{i\in I}$  be a sequence of $K$-vector spaces indexed by objects in $I$. Functors of the form 
$\oplus_{i\in I} K(i,-)\otimes V_i$ are called {\bf free}.
Two free functors $\oplus_{i\in I} K(i,-)\otimes V_i$ and $\oplus_{i\in I} K(i,-)\otimes W_i$  are isomorphic if and only if, for any $i$ in  $I$, the vector spaces $V_i$ and $W_i$ are isomorphic.
Let $F= \oplus_{i\in I} K(i,-)\otimes V_i$ be a free functor. The vector spaces  $V_i$ are called the {\bf components} of $F$.  If all component are finite dimensional, then  $F$   is called of {\bf finite type}. 
 The {\bf support} of a free functor $\oplus_{i\in I} K(i,-)\otimes V_i$  is the subset  of the set of objects  of $I$ consisting of those  $i$ in $I$  for which  $V_{i}\ne 0$.  A free functor 
is said to be of {\bf finite rank} if has finite support and is of finite type. If $F=\oplus_{i\in I} K(i,-)\otimes V_i$ is of finite rank, the number $\text{rank}(F):=\sum_{i\in I} \text{dim}_K V_i$ is called the {\bf rank} of F.  

Consider a  free functor $F=\oplus_{i\in I} K(i,-)\otimes V_i$ of finite type.  
The  $0$-{\bf Betti diagram} of  $F$ is  defined to be the multiset  on the set of objects of $I$ given by $\beta_0F(i):=\text{dim}_KV_i$.
The $0$-Betti diagram of a free finite type functor determines its isomorphism type.
Note that if $F$ is free and of finite rank, then the multiset $\beta_0F$ is  finite of size $\text{rank}(F)$.
\end{point}



\begin{point}\label{point minimal}
A morphism $\phi\colon X\to Y$ in a category $\calc$ is called {\bf minimal} if any morphism  $f\colon X\to X$ satisfying   $\phi=\phi f$ is an isomorphism. A natural transformation $\phi\colon F\to G$   in $\text{Fun}(I,\text{Vect}_K)$
is called a {\bf minimal cover}  of $G$, if $F$  is free and $\phi$ is both minimal and an epimorphism.
Minimal covers 
are unique up to isomorphism: if  $\phi\colon F\to G$ and  $\phi'\colon F'\to G$
are minimal covers  of $G$, then there is an isomorphism (non necessarily unique) $f\colon F\to F'$ such that
$\phi=\phi' f$. Furthermore any  $g\colon F\to F'$  for which $\phi=\phi' g$ is an isomorphism (minimality).

Consider a functor $G:I\rightarrow \text{Vect}_K$ that admits a     minimal cover $\phi \colon F\to G$.
If $F$ is of finite type, then we say that $G$ is of {\bf finite type}. 
If $F$ is of  finite rank, then we say that $G$  is of  {\bf finite rank}  and define   the {\bf rank} of $G$ to be the rank of the free functor $F$ and denote it by 
$\text{rank}(G)$.  
We define the {\bf support} of $G$ to be the support of $F$ and denote it by $\text{supp}(G)$.
Note that $G$ is of finite rank if and only if it has finite support and is of finite type.
If $G$ is of finite type,
we define the {\bf $0$-Betti diagram} of $G$ to be the multiset on the set of objects of $I$ given by
the $0$-Betti diagram of the free finite type functor $F$ (see~\ref{point freefunctors}) and denote it by $\beta_0 G$.
Being of finite type, of finite rank, and  the  invariants $\text{rank}(G)$, $\text{supp}(G)$, and $\beta_0 G$
do not depend on the choice of the minimal cover of $G$.

Consider a functor $G:I \to \text{Vect}_K$, recall that an element $g$ in $G(i)$ induces a unique natural transformation,
denoted by the same symbol  $g\colon K(i,-)\to G$, that maps the element   $\text{id}_i$ in $K\text{mor}_I(i,i)=K(i,i)$ to $g$ in $G(i)$.
A {\bf minimal set of generators} for $G$  is a sequence  of elements $\{g_1 \in G(i_1),\ldots, g_n\in G(i_n)\}$  such that the induced natural transformation $\oplus_{k=1}^n g_k\colon \oplus_{k=1}^n K(i_k,-)\to G$ is a minimal cover. A functor has a minimal set of generators if and only if it is of finite rank, in which case the number
of minimal generators is given by $\text{rank}(G)$. If $\{g_1 \in G(i_1),\ldots, g_n\in G(i_n)\}$  is a minimal set of generators of $G$, then the multiset $\beta_0G\colon I\to \N$ assigns to an object $i$ in $I$ the number of generators that belong to $G(i)$.
\end{point}

\begin{point}\label{point NRposets} Let $r$ be a positive natural number and  $v=(v_1,\ldots,v_r)$ and $w=(w_1,\ldots,w_r)$ be   $r$-tuples of {\em non negative} rational numbers. Define: 
\begin{itemize}
\item  $\|v-w\|:=\text{max}\{|v_i-w_i|\ |\ 1\leq i\leq r\}$,
\item $v\leq w$ if  $v_i\leq w'_i$ for any $i$,
\item  $v<w$ if $v\leq w$ and $v\not=w$. 
\end{itemize}
We call the number $||w||$ the {\bf norm} of $w$.
The relation $\leq$ is a partial order. We use the symbol  $\QQ^{r}$ to denote the  category  associated to this poset, i.e., the category whose objects are
$r$-tuples of non negative rational numbers   with the sets of morphisms  $\text{mor}_{\QQ^r}(v,w)$ being either empty
if $v\not\leq w$, or consisting of  only one element in the case $v\leq w$. 
Note that if  $v\leq w\leq u$ in $\QQ^r$, then $||v-w||\leq ||v-u||$. The full subcategory of $\QQ^{r}$  whose objects are $r$-tuples of natural numbers is denoted by $\N^r$. Both posets $\QQ^r$ and $\N^r$ are lattices.
This means that for any finite set of elements
$S$ in $\QQ^r$ (respectively $\N^r$), there are elements $\text{meet}(S)$ and $\text{join}(S)$ in $\QQ^r$ (respectively $\N^r$) 
with the following properties. First, for any $v$ in $S$, $\text{meet}(S)\leq v\leq \text{join}(S)$. Second, if $u$ and $w$ are elements in $\QQ^r$ (respectively $\N^r$) for which $u\leq v\leq w$, for any $ v$ in $S$, then $u\leq \text{meet}(S)$ and $\text{join}(S)\leq w$. Observe that the elements $\text{meet}(S)$  and  $\text{join}(S)$  may not   belong to $S$. 

Let $S$ be a subset in  $\QQ^r$ (respectively $\N^r$). An element $v$ in $S$ is called {\bf minimal} if, for any $w<v$, $w$ does not belong to $S$.  The set of minimal elements of any non-empty subset of $\N^r$ is never empty and is finite. Neither of these statements are true for $\QQ^r$.

The element in $\QQ^r$ whose  coordinates are all $0$ is called either the {\bf origin} or the {\bf zero} element and is denoted simply by $0$.
The element in $\QQ^r$ whose  coordinates are all $0$ except  the $i$-th one which is $1$ is called the  $i$-th {\bf standard vector} and denoted  by $e_i$. 
\end{point}

\begin{point}\label{pt cone}
The set of all linear combinations of 
elements  $g_1,\ldots, g_n$ in $\QQ^r$  with non-negative rational coefficients is called the {\bf cone generated} by $g_1,\ldots, g_n$ 
 and denoted by $\text{Cone}(g_1,g_2\ldots, g_n)$. 
 A {\bf cone} in $\QQ^r$ is by definition a subset of $\QQ^r$  of the form $\text{Cone}(g_1,g_2\ldots, g_n)$ for some
 non-empty sequence of elements $g_1,\ldots, g_n$ in $\QQ^r$. 
 A {\bf ray} is a cone in $\QQ^r$ generated by one non-zero element.
 \end{point}

\begin{point} \label{pt translation}
Let $w$ be in $\QQ^r$ and $\calc$ be a category. Consider
the functor $-+w\colon \QQ^r\to \QQ^r$ that maps $v\leq u$   to $v+w\leq u+w$. The composition of
  $-+w$  with a functor $F\colon \QQ^r\to \calc$ is called the $w$-{\bf translation} of $F$ and is denoted by
$F(-+w)\colon  \QQ^r\to \calc$.  Let $\tau$ be in $\QQ^r$. Two functors $F,G$ in $\text{Fun}(\QQ^r,\calc)$ are 
$\tau$-{\bf interleaved} if there exist natural transformations $\phi:F \rightarrow G(-+\tau)$ and 
$\psi: G \rightarrow F(-+\tau)$ such that $\psi_{v+\tau} \circ \phi_{v} = F(v\leq v+2\tau)$ and 
$\phi_{v+\tau} \circ \psi_{v} = G(v\leq v+2\tau)$ for any $v$ in $\QQ^r$.
This definition follows the definition of interleaving given in \cite{generlized_interleavings} for functors indexed by a preordered set.
\end{point}

\begin{point}\label{pt metric}
The symbol ${\mathbf R}_{\infty}$  denotes the poset whose underlying set is the disjoint union of the set of non-negative real numbers and the singleton $\{\infty\}$.  The order on  ${\mathbf R}_{\infty}$ is given by the usual being smaller or equal relation for non-negative real numbers and assuming that $x\leq \infty$ for any $x$ in  ${\mathbf R}_{\infty}$.

An {\bf extended pseudometric}  on a set $X$ is a function $d\colon X\times X\to {\mathbf R}_{\infty}$
such that:
\begin{enumerate}
\item for any $x$ and $y$  in $X$,  $d(y,x)=d(x,y)$; 
\item  for any $x$ in $X$, $d(x,x)=0$;
\item  for any $x$, $y$, and $z$ in $X$, $d(x,z)\leq d(x,y)+d(y,z)$.  
 \end{enumerate}
A set  equipped with an extended pseudometric  is called 
an extended pseudometric space.

Let $d$ be an extended pseudometric  on a set $X$. For any positive real number $t$ and any $x$ in $X$,
$B(x,t)$ denotes the subset of $X$ consisting of these elements $y$ for which $d(x,y)<t$. This subset is called
the open ball around $x$ with radius $t$. These sets form a base of a topology on $X$.  

Let $k$ be a positive real number.
Given two extended pseudometric spaces $(X, d_X)$ and $(Y, d_Y)$, a function $f : X\rightarrow Y$ is called 
$k$-Lipschitz if  $d_Y(f(x_1), f(x_2)) \leq k d_X(x_1, x_2)$ for any
$x_1,\,x_2$ in $X$.
 \end{point}
 
\section{Functors indexed by $\N^{r}$}\label{sec:funcindbyN}
In this section we recall how to determine that  a functor $F\colon \N^{r}\to \text{Vect}_K$
is of finite rank and how to, for a such a  functor,   compute its  support, rank and  $0$-Betti diagram. 
 The idea of using the $0$-Betti diagram as an informative invariant in the context of multidimensional persistence was first introduced in \cite{Multi}.

We will also recall the classification of finite rank functors in the case $r=1$.  These are
standard results as the category $\text{Fun}(\N^{r},\text{Vect}_K)$ is equivalent to the category of $r$-graded modules over the  polynomial ring in $r$ variables  with the standard $r$-grading  (see  \cite{Multi}).
In the appendix  we present   a classical way of analysing 
  basic   properties of the category  $\text{Fun}(\N^{r},\text{Vect}_K)$, where we in particular  identify its compact and  projective objects, and discuss minimal covers.  We do that for self containment of the paper and
  to illustrate that this material including  all the proofs  and the   classification of compact functors for $r=1$ can be presented on less than 4 pages.
  
 The  radical of a functor $G\colon \N^{r}\to \text{Vect}_K$ is  the key tool to determine its rank, support, and 
  $0$-Betti diagram. Recall that the  radical of $G$  is a subfunctor $\text{rad}(G)\subset G$ 
 whose value $\text{rad}(G)(v)$ is  the subspace of $G(v)$ given by the sum
of all the images of $G(u< v)\colon G(u)\to G(v)$ for all $u< v$ (see~\ref{pt radical}).  
The quotient $G/\text{rad}(G)$
is isomorphic to a functor of the form  $\oplus_{v\in \N^r} (U_v\otimes V_v)$,
where $\{V_v\}_{v\in \N^r}$ is a sequence of vector spaces and 
$U_v\colon \N^{r}\to \text{Vect}_K$ is  the unique functor such that $U_v(v)=K$ and $U_v(w)=0$ if $w\not =v$
(see~\ref{pt semisimpl}).  We can now state (for the proofs see~\ref{prop fdimvalues} and~\ref{pt invsemisimple}):
\begin{itemize}
\item Any  functor in  $\text{Fun}(\N^{r},\text{Vect}_K)$ admits a minimal cover.
\item $G$ is of finite type if $V_v$ is finite dimensional for any $v$.
\item $\text{supp}(G)=\{v\in \N^{r}\ | V_v\not=0\}$.
\item $G$ is of finite rank if and only if $\{v\in \N^{r}\ | V_v\not=0\}$ is a finite set and $V_v$ is finite dimensional for any $v$,
in which case $\text{rank}(G):=\sum_{v\in \N^{r}} \text{dim}_K V_v$.
\item $\beta_0G(v)=\text{dim}_K V_v $.
\end{itemize}

Let $w\leq u$ be in $\N$. Recall that the bar  starting in $w$ and ending in $u$ is a functor 
$[w,u)\colon \N\to \text{Vect}_K$ given by the cokernel of the unique inclusion $K(u,-)\subset K(w,-)$ (see~\ref{pt compactr=1}). The classification of finite rank functors in $\text{Fun}(\N,\text{Vect}_K)$ states (~\ref{prop classification}):
 \begin{itemize}
\item Any functor of finite rank $F\colon\N\to \text{Vect}_K$ is isomorphic to a finite direct sum of functors of the form $[w,u)$ and $K(v,-)$. 
Moreover the isomorphism types of these   summands are uniquely determined by the isomorphism type of  $F$.
 \end{itemize}
 The above theorem, also known as the structure theorem for finitely generated graded modules over PID's, allows us to decompose any functor of finite rank $F\colon\N\to \text{Vect}_K$ as a direct sum of bars  and free functors. Such decomposition in persistent homology is commonly visualised through a barcode where each bar represents an indecomposable summand (see~\cite{persistence}).

  \section{Tameness}\label{sec tamness}

In this section we introduce the category of tame functors indexed by $\QQ^r$. Intuitively, a tame functor is an extension of a functor indexed by $\N^r$ to a functor indexed by $\QQ^r$, which is constant on regions we call fundamental domains.
In the following we will be particularly interested in tame and compact functors with values in $\text{Vect}_K$ (for compactness see \ref{pt compact}).
On one side,   computing homological invariants of such functors, can be recasted to computing analogous  invariants of  functors indexed by $\N^r$. On the other,  the indexing category $\QQ^r$ offers new ways of comparing and measuring distances between tame functors.  
One could possibly use real number and define tame functors to be indexed by ${\mathbf R}^r$. For  technical reasons however we decided to use the  rational numbers $\QQ$.  This in our view is not a restrictive choice, as many functors are coming from  data sets which are obtained by  incremental 
and discrete measurements, for instance see the examples presented in \cite{Multi}.

\begin{point}{\bf Fundamental domain.}\label{pt fund dom}
If a group acts on a topological space $X$, a connected subspace of $X$ which contains exactly one  representative of each orbit  is called a fundamental domain. For example  the product of half open intervals $[0,1)^r$ is an example of a fundamental domain of 
 the action of $\ZZ^r$ on  $\mathbf{R}^r$ by translations. For studying multidimensional persistence we would like to replace the group  with the monoid $\NN^r$ acting by various  translations on $\QQ^r$.
 In this article we are interested in   actions  given by the following different embeddings of  $\NN^r$ in  $\QQ^r$.
Let $\alpha$ be a positive rational number and let the same symbol  $\alpha\colon {\N}^r\to \QQ^r$
denote  the unique functor that maps an object $w$ in  ${\N}^r$ to $\alpha w$ (the multiplication of all the coordinates of $w$ by $\alpha$) in $\QQ^r$.
Then for  $v$ in $\QQ^r$, consider the following finite subset
$B_{\alpha}v:=\{ w\in \N^r \ |\  \alpha w\leq v\}$  of $\N^r$ and define $b_\alpha v:=\text{join}\, B_{\alpha}(v)$ (see~\ref{point NRposets}).  
For any $w$ in $\N^r$ we call the subset $D_{\alpha}w=\{v\in \QQ^r\ |\ b_\alpha v=w\}\subset \QQ^r$ the {\bf fundamental domain} of ${\alpha}$ at $w$.
\end{point}

\begin{example}\label{ex:fund dom}
Let $\alpha=1$, so that $\alpha \colon \NN^2 \to \QQ^2$ is the standard inclusion $x\mapsto x$. Then for any ${v\in [2,3)\times[2,3)\subset\QQ^2}$, we have that $B_\alpha v=\{w\in \NN^2 \mid w\le v\}=\{(x,y)\in \NN^2 \mid x,y\le 2\}$ and $b_\alpha v=(2,2)$. Conversely the fundamental domain (of $1$) at $(2,2)$ is $D_1(2,2)=[2,3)\times [2,3)\subset \QQ^2$. This is indiacated in Figure~\ref{fig:funddom1}.


\vspace{10pt}
\begin{minipage}{\textwidth}
\begin{center}
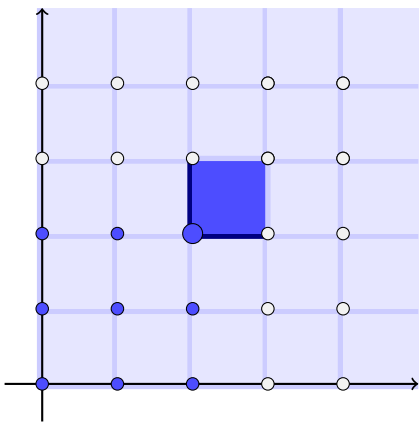

\begin{tikzpicture}[scale=1]

\foreach \i in {0,...,4}{
\foreach \j in {0,...,4}{
\draw[blue!20!white, line width=4pt] (\i+1,\j)--(\i,\j)--(\i,\j+1);
\draw [draw=none,fill=blue!10!white] (\i,\j)--(\i+1,\j)--(\i+1,\j+1)--(\i,\j+1)--(\i,\j);
}}

\draw[blue!50!black,line width=4pt] (3,2)--(2,2)--(2,3);
\draw [draw=none,fill=blue!70!white] (2,2)--(3,2)--(3,3)--(2,3)--(2,2);
\draw[blue!20!white, line width=2pt] (2,3)--(3,3)--(3,2);
\draw[thick,->] (-0.5,0)--(5,0);
\draw[thick,->] (0,-0.5)--(0,5);
\foreach \i in {0,...,4}{
\foreach \j in {3,4}{
\node[fill=black!5!white, shape=circle, scale=0.5,draw=black] at (\j,\i) {};
\node[fill=black!5!white, shape=circle, scale=0.5,draw=black] at (\i,\j) {};
}}
\foreach \x in {0,...,2}{
	\foreach \y in {0,...,2}{
		\node[fill=blue!70!white, shape=circle, scale=0.5,draw=black] at (\x,\y) {};}}
\node [fill=blue!70!white, shape=circle, scale=0.8, draw=black] at (2,2) {};

\end{tikzpicture}
\captionof{figure}{$D_1(2,2)$, $B_1 v$ and $b_1v$ for $v\in D_1(2,2)$}\label{fig:funddom1}
\end{center}
\end{minipage}
\end{example}
%
%

Note that we have the following properties:
%
\begin{enumerate}
\item the fundamental domain $D_{\alpha}w$ consists of  elements $v$ in $\QQ^r$ such that $\alpha w\leq v$ and $||v-\alpha w||<\alpha$ i.e. all the coordinates of $v-\alpha w$ are non-negative and are strictly smaller than $\alpha$;
\item for any $v$ in $\QQ^r$, $\alpha b_{\alpha} v\leq v$ and $||v-\alpha b_{\alpha}v||<\alpha$;
\item  for any  $w$  in $\N^r$,  $w=b_{\alpha}\alpha w$;
\item if $v\leq u$ in $\QQ^r$, then $B_{\alpha}v\subset B_{\alpha}u$ and hence
$b_{\alpha}v\leq b_{\alpha}u$ in $\N^r$.
\end{enumerate}

We use these properties to extend functors indexed by $\NN^r$ to functors indexed by $\QQ^r$ along the embedding $\alpha\colon \NN^r\to \QQ^r$.  For $F\colon\NN^r \to \calc$, define a new functor $\alpha^! F\colon \QQ^r\to \calc$, by letting $\alpha^!F(v):=F(b_\alpha v)$ for $v$ in $\QQ^r$ and setting   ${\alpha}^{!}F(v\leq u)\colon {\alpha}^{!}F(v)\to {\alpha}^{!}F(u)$ to be  
$F(b_\alpha v\leq b_\alpha u)$. By construction this functor is constant on all the fundamental domains, i.e.,
$F(b_{\alpha}v\leq v)$ is the identity for any $v$ in $\QQ^r$.

\begin{example}
Consider the semi-simple functor $U_{(2,2)}\colon \NN^2 \to \Vect_K$ (see \ref{pt semisimpl}). By definition the value of  $1^! U_{(2,2)}\colon \QQ^2\to \Vect_K$ is $K$ on the fundamental domain $D_1(2,2)=[2,3)\times [2,3)$ and zero otherwise. Furthermore 
$1^! U_{(2,2)}(v\le u)$ is the identity morphism if $b_1 v=b_1 u=(2,2)$ and the zero morphism otherwise. 
\end{example}

We claim that the construction $(F\colon {\N}^r\to \calc)\mapsto  ({\alpha}^{!}F\colon \QQ^r\to \calc)$ is natural.  To see this let the restriction functor along $\alpha$  be denoted by  
${\alpha}^{\ast}\colon \text{Fun}(\QQ^r,\calc)\to  \text{Fun}({\N}^r,\calc)$. It assigns to a functor $F\colon\QQ^r\to \calc$ the composition $F\alpha\colon\N^r\to \calc$. 
Note that, for any natural transformation $\psi\colon F\to G$ in
$ \text{Fun}({\N}^r,\calc)$, there is a unique natural transformation  ${\alpha}^{!}\psi\colon
{\alpha}^{!}F\to {\alpha}^{!}G$ for which ${\alpha}^{\ast}{\alpha}^{!}\psi=\psi$ (as ${\alpha}^{!}F$ is constant on the fundamental domains).  Because of this uniqueness it is clear that $\alpha^{!}(\psi\phi)=
(\alpha^{!}\psi)(\alpha^{!}\phi)$ and so $\alpha^{!}$ is a functor.

Let $F\colon\NN^r\to \calc$ be a functor.
Since $\alpha^{!}F$ is constant on fundamental domains, a natural transformation
$\psi\colon \alpha^{!}F\to G$ into any functor $G\colon\QQ^r\to \calc$, is uniquely determined by
its restriction $\alpha^{\ast}\psi\colon F=\alpha^{\ast}\alpha^{!}F\to \alpha^{\ast}G$. This gives a bijection between $\text{Nat}_{\QQ^r}(\alpha^{!}F,G)$ and $\text{Nat}_{\NN^r}(F,\alpha^{\ast}G)$.
In categorical terms this means that $\alpha^{!}\colon \text{Fun}({\N}^r,\calc)\to \text{Fun}(\QQ^r,\calc)$ is left adjoint to the restriction functor ${\alpha}^{\ast}\colon \text{Fun}(\QQ^r,\calc)\to  \text{Fun}({\N}^r,\calc)$. Recall that such left adjoints are also called left Kan extensions (see~\cite{Categories}).
In particular, since ${\alpha}^{\ast} {\alpha}^{!}G=G$ for any $G\colon\N^r\to\calc$, the 
  function  
$(\psi\colon F\to G)\mapsto ({\alpha}^!\psi\colon {\alpha}^!F\to {\alpha}^!G)$ is a bijection between $\text{\rm Nat}(F,G)$ and  $\text{\rm Nat}( {\alpha}^!F,  {\alpha}^!G)$.

 Let 
$G\colon \QQ^r\to \text{Vect}_K$ be a functor. The  natural transformation  adjoint to the identity $\text{id}\colon {\alpha}^{\ast}G\to {\alpha}^{\ast}G$ is denoted by $\omega\colon {\alpha}^{!}{\alpha}^{\ast}G\to G$.
Explicitly, for $v$ in $\QQ^r$,  the morphism $\omega_v\colon {\alpha}^{!}{\alpha}^{\ast}G(v)=G(\alpha b_{\alpha}v)\to G(v)$ is given by  $G(\alpha b_{\alpha}v\leq v)$.

\begin{Def}\label{def tame}
Let $\alpha$ be in $\QQ$. A functor $G\colon\QQ^r\to \calc$ is called {\bf $\alpha$-tame} if it is isomorphic to a functor of the form $\alpha^!F$ for some $F\colon \N^r \to \calc$. A functor is {\bf tame} if it is $\alpha$-tame for some $\alpha \in \QQ$.

We will  use the symbol  $\text{\rm Tame}(\QQ^r,\calc)$ to denote the full subcategory of $\text{\rm Fun}(\QQ^r,\calc)$ whose objects are tame functors.
\end{Def}

Note $G\colon \QQ^r\to \calc$ is $\alpha$-tame  if and only if  
$\omega\colon {\alpha}^{!}{\alpha}^{\ast}G\to G$ is an isomorphism, i.e., if 
 $G(\alpha b_{\alpha}v\leq v)$ is an isomorphism for any $v$ in $\QQ^r$. Furthermore $G\colon \QQ^r\to \calc$ is $\alpha$-tame if and only if $G(\alpha w\leq v)$   is an isomorphism for any $w$ in $\N^r$ and $v$ in $\QQ^r$ such that $\alpha w\leq v$ and $||v-\alpha w||<\alpha$. In other words, tame functors are exactly the functors that are constant on fundamental domains.

 Note that the following diagram commutes for any positive rational $\alpha$:
 \[\xymatrix{
 {\N}^r\rto^-{n} \ar@/_15pt/[rr]|{{\alpha}}&  {\N}^r\rto^-{{\alpha/n}} & \QQ^r
 }\] 
 Thus the functors  ${\alpha}^! F$ and $({\alpha/n})^! n^! F$ are naturally isomorphic proving:
 \begin{prop}\label{prop divtame}
 If $G\colon \QQ^r\to {\mathcal C}$ is $\alpha$-tame, then it is also $\alpha/n$-tame for any positive natural number $n$.
 \end{prop}
 
 Another operation on functors that preserve tameness is translation. Let $w$ be in $\QQ^r$. Recall that 
  the $w$-translation of $F\colon \QQ^r\to \calc$, denoted by $F(-+w)$,  is the composition of $F$ with 
the functor $-+w\colon \QQ^r\to \QQ^r$ that maps $v\leq u$   to $v+w\leq u+w$ (see~\ref{pt translation}).

\begin{prop}
If $F\colon \QQ^r\to \calc$ is tame, then so is $F(-+w)$ for any $w$ in  $\QQ^r$.
\end{prop}
\begin{proof}
Assume  $F$  is $\alpha$-tame. Consider $w=(w_1,\ldots, w_r)$ in $\QQ^r$.
Since $\alpha$ and the coordinates of $w$ are rational, there are natural numbers $m$ and $n_1,\ldots, n_r$ such that $\alpha/m=w_1/n_1=\cdots w_r/n_r=\mu$.  We claim that $F(-+w)$ is $\mu$-tame. For that we need to show that,  for any $v$ in $\QQ^r$, 
$F( \mu b_{\mu}v+w\leq v+w)$ is an isomorphism. Set $u:=(n_1,\ldots, n_r)$ in $\N^r$.
Note  that $w=\mu u$, and hence $b_{\mu}(v+w)=b_{\mu}v + u$, which implies
$\mu b_{\mu}(v+w)=\mu b_{\mu}v +w$.  The morphism $F( \mu b_{\mu}v+w\leq v+w)$
is then an isomorphism since $F$ is also $\mu$-tame (see~\ref{prop divtame}).
\end{proof}

\section{The category $\text{\rm Tame}(\QQ^r,\text{\rm Vect}_K)$}\label{sec tamevector}
In this section we describe basic properties of the category $\text{\rm Tame}(\QQ^r,\text{\rm Vect}_K)$ including
invariants called the $0$-Betti diagrams. For that we need  to discuss  the restriction $\alpha^{\ast}$ and the Kan extension $\alpha^!$ for functors with values in $\text{\rm Vect}_K$ (similar properties hold for functors  with values in any abelian category).   Note  that  Proposition~\ref{prop kercokertoam} and Corollary~\ref{cor tamnessprop} are  false if $\text{\rm Tame}(\QQ^r,\text{\rm Vect}_K)$ is replaced by $\text{\rm Tame}(\RR^r,\text{\rm Vect}_K)$.
\begin{prop}\label{prop basickan}
Let $\alpha$    be a positive rational number.
\begin{enumerate}
\item  
The left  Kan extension
${\alpha}^!K(v,-)\colon \QQ^r\to\text{\rm Vect}_K$ of  $K(v,-)\colon \N^r\to\text{\rm Vect}_K$  is  isomorphic to $K(\alpha v,-)$ and hence is free.
\item The  restriction of  $K(v,-)\colon \QQ^r\to\text{\rm Vect}_K$ along $\alpha\colon\N^r\to\QQ^r$
is also free and isomorphic to $K(\rm{meet} \{w\in \N^r\ |\ v\leq \alpha w\},-)$. 
 \item Both functors ${\alpha}^{\ast}\!: \text{\rm Fun}(\QQ^r,\text{\rm Vect}_K)\rightleftarrows  \text{\rm Fun}({\N}^r,\text{\rm Vect}_K):\! {\alpha}^!$  commute with  arbitrary colimits and in particular with direct sums.
 \item  A sequence of morphisms in  $\text{\rm Fun}({\N}^r,\text{\rm Vect}_K)$ is exact if and only if 
 ${\alpha}^!$ transforms it into an exact  sequence  in $\text{\rm Fun}(\QQ^r,\text{\rm Vect}_K)$.
 \item   If a sequence of morphisms in  $\text{\rm Fun}(\QQ^r,\text{\rm Vect}_K)$ is exact, then so is
 its restriction via ${\alpha}$ in $\text{\rm Fun}({\N}^r,\text{\rm Vect}_K)$.
\item A functor $F\colon \N^r\to \text{\rm Vect}_K$ is free if and only if $\alpha^!F \colon \QQ^r\to  \text{\rm Vect}_K$ is free.
\item If  $\phi \colon F\to G $ is  a minimal cover in $\text{\rm Fun}( \N^r , \text{\rm Vect}_K)$, then  $\alpha^! \phi \colon \alpha^! F\to \alpha^! G $ is a minimal cover  in $\text{\rm Fun}(\QQ^r,\text{\rm Vect}_K).$
\item If $F$ in $\text{\rm Fun}(\N^r,\text{\rm Vect}_K)$ is compact (see~\ref{pt compact}), then so is the functor $\alpha^! F$ in
$\text{\rm Fun}(\QQ^r,\text{\rm Vect}_K)$.
\item Let  $F\colon \QQ^r\to \text{\rm Vect}_K$ be $\alpha$-tame.
Then $F$ is compact in  $ \text{\rm Fun}(\QQ^r,\text{\rm Vect}_K)$ if and only if
$\alpha^{\ast}F$ is compact in $\text{\rm Fun}(\N^r,\text{\rm Vect}_K)$.
\item Let  $0\to F\to G\to H\to 0$ be   an exact sequence of tame functors in $ \text{\rm Fun}(\QQ^r,\text{\rm Vect}_K)$. Then $G$ is compact if and only if $F$ and $H$ are compact.  
\end{enumerate}
\end{prop}

\begin{proof}
\noindent 
Statement (1) and (2) are clear.    Statement (3) follows from the construction and the fact that colimits  in functor categories are formed object-wise. Same argument gives (4) and (5).  
Statement (6) is implied by (1), (2), and (3). 
To prove (7) note that by  (6) and (4) we know   $\alpha^! F$ is free and  $\alpha^! \phi: \alpha^! F\to \alpha^! G$ is an epimorphism.  
Minimality of $\alpha^! \phi$ follows from the fact that $\alpha^! $ induces a bijection between $\text{Nat}(F,F)$  and
$\text{Nat}(\alpha^! F,\alpha^! F)$. 
Since the same argument  can be used  to  prove both (8) and (9), we present the details of  how to show  (8) only.  Consider a compact functor $F$  in $\text{\rm Fun}(\N^r,\text{\rm Vect}_K)$  and a sequence of subfunctors 
$A_0\subset A_1\subset\cdots \subset \alpha^!F$ in $ \text{\rm Fun}(\QQ^r,\text{\rm Vect}_K)$ such that
$\text{colim}\, A_i=\alpha^!F$. By taking the restriction along $\alpha\colon\N^r\to\QQ^r$ and using (3) we obtain a filtration $\alpha^{\ast}A_0\subset \alpha^{\ast}A_1\subset \cdots\subset \alpha^{\ast}\alpha^!F=F$
such that $\text{colim}(\alpha^{\ast}A_i)=F$.  As $F$   is compact, there is $n$ such that $\alpha^{\ast}A_n=F$.
Apply the Kan extension to get
$\alpha^!\alpha^{\ast} A_n\to A_n\subset \alpha^!F$. The composition of these two natural transformations is an isomorphism. It follows  $A_n= \alpha^!F$ and consequently $\alpha^!F$ is  compact. Statement (10) follows from (9) and Proposition \ref{prop comactfunctors}.

 \end{proof}

\begin{prop}\label{prop kercokertoam}
Let $\phi\colon F\to G$ be a natural transformation  in  $\text{\rm Fun}(\QQ^r,\text{\rm Vect}_K)$. If $F$ and $G$ are tame, then so are $\text{\rm ker}(\phi)$ and $\text{\rm coker}(\phi)$.
\end{prop}
\begin{proof}
Let  $F$ be $\alpha$-tame and $G$ be   $\beta$-tame. Since  $\alpha$ and $\beta$ are rational numbers, there are natural numbers $m$ and $n$ such that $\alpha/n=\beta/m$.
The functors $F$ and $G$ are therefore  $\mu=\alpha/n$-tame (see~\ref{prop divtame}).
 Since Kan extensions  preserve exactness (see~\ref{prop basickan}), $\text{\rm ker}(\phi)$
is isomorphic to ${\mu}^!(\text{ker}({\mu}^{\ast}\phi ))$ and 
$\text{\rm coker}(\phi)$ is isomorphic to ${\mu}^!(\text{coker}({\mu}^{\ast}\phi ))$.
\end{proof}

As a Corollary of~\ref{prop kercokertoam}, we get:
\begin{cor}\label{cor tamnessprop}$ $

\begin{enumerate}
\item Consider an exact sequence in $\text{\rm Fun}(\QQ^r,\text{\rm Vect}_K)$:
\[0\ra F\ra G\ra H\ra 0\]
If two out of $F$, $G$, and $H$ are tame, then so is the third.
\item If $F$ and $G$  in $\text{\rm Fun}(\QQ^r,\text{\rm Vect}_K)$ are tame, then so is $F\oplus G$.
\item $\text{\rm Tame}(\QQ^r,\text{\rm Vect}_K)$ is an abelian subcategory of $\text{\rm Fun}(\QQ^r,\text{\rm Vect}_K)$.
\end{enumerate}
\end{cor}

Corollary~\ref{cor tamnessprop} states in principle that 
$\text{\rm Tame}(\QQ^r,\text{\rm Vect}_K)$ is an abelian category. Note that even though 
$\text{\rm Tame}(\QQ^r,\text{\rm Vect}_K)$ is
 closed under finite direct sums, 
infinite direct sums however do not  in general preserve tameness.



We finish the section by explaining how to compute the support, rank and $0$-Betti diagram of a
tame functor $G\colon \QQ^r \to \text{\rm Vect}_K$.
Here is the procedure:
\begin{itemize}
\item Choose $\alpha$ in $ \QQ$, such that $\omega:\alpha^! \alpha^{\ast}G\to G$ is an isomorphism. In this step we choose a scale $\alpha$ for which $G$   is $\alpha$-tame.
\item 
Find a sequence of vector spaces $\{V_w\}_{w\in\N^r}$ such that  $\alpha^{\ast}G/\text{rad}(\alpha^{\ast}G)$ is isomorphic to $\oplus_{w \in \N^r} U_w \otimes V_w $.
\end{itemize}
We have now all the needed information to  compute $\text{supp}(G), \text{rank}(G)$ and $\beta_0G$:
\begin{prop}\hspace{1mm}
Let $\alpha$ and $\{V_w\}_{w\in\N^r}$ be defined as above.
\begin{enumerate}
\item $\text{\rm supp}(G)=\{\alpha w\ |\ w\in\text{\rm supp}(\alpha^{\ast}G)\}=\{\alpha w\ |\  w\in N^r\text{ and } V_w\not=0\}$;
\item $\text{\rm rank}(G)=\text{\rm rank}(\alpha^{\ast}G)=\sum_{ w\in \N^r}\text{\rm dim}_K V_w$;
\item $\beta_0G\colon \QQ^r\to \N$ is given by:
\[\beta_0G(v)=\begin{cases}
\beta_0(\alpha^{\ast}G)(w)=\text{\rm dim}_KV_{w} &\text{ if  }v=\alpha w\text{ for }w\in\N^r\\
0 &\text{ otherwise }
\end{cases}\]
\end{enumerate}
 \end{prop}
\begin{proof}
This is a consequence of two facts: first $\alpha^! K(w,-)=K(\alpha w,-)$ (see~\ref{prop basickan}) and  second  if $F\to \alpha^{\ast}G$   is a minimal cover of $ \alpha^{\ast}G$
in $\text{Fun}(N^r,\text{Vect}_K)$, then $\alpha^! F\to \alpha^!\alpha^{\ast}G=G$ is a minimal cover of
$G$  (see~\ref{prop basickan}.(7)).
\end{proof}

The right sides of the equalities in the above proposition a priori depend on the choice of a scale $\alpha$ for which 
$G$   is $\alpha$-tame. However since the left sides are independent  of $\alpha$, then so are the right sides.

\begin{example}\label{ex bar}
Let $w\leq u$ be two elements in $\QQ^r$. There is a unique inclusion $K(u,-)\subset K(w,-)$.
The cokernel of this inclusion is denoted by $[w,u)$.
Numerical invariants for functors of this type are studied in ~\cite{Bar}. 
 Since the free functors are tame, according to~\ref{cor tamnessprop}.(1) $[w,u)$ is tame. It is clear that $[w,u)$ is also compact. 
Note that $\text{supp}([w,u))=\{w\}$, $\text{rank}([w,u))=1$, and:
 \[\beta_0[w,u)(v)=\begin{cases}
 1 &\text{ if } v=w\\
 0 &\text{ if } v\not =w.
 \end{cases}\]
\end{example} 
Similarly to functors indexed by $\N$ (see~\ref{pt compactr=1}),  there is a classification for compact and tame functors indexed by $\QQ$.
 \begin{prop}\label{prop chartamer=1}
  Any compact object in $\text{\rm Tame}(\QQ, \text{\rm Vect}_K)$  is isomorphic to a finite direct sum of functors of the form
 $[w,u)$ and $K(v,-)$.  Moreover the isomorphism types of  these   summands are uniquely determined by the isomorphism type of the functor.
 \end{prop}
 \begin{proof}
 Let $G\colon \QQ\to \text{\rm Vect}_K$ be a compact and tame functor. Choose $\alpha$ in $\QQ$ such that
 $G=\alpha^{!}\alpha^{\ast}G$. Since $\alpha^{\ast}G\colon\N\to\text{\rm Vect}_K$ is compact, it is isomorphic to
 a finite direct sum of bars and free functors (see~\ref{pt compactr=1}). As $\alpha^{!}$ commutes with direct sums, we get
 the desired decomposition of $G$. Uniqueness is shown in the same way.
 \end{proof}

Note that Proposition~\ref{prop chartamer=1} is a direct extension of the classical classification theorem of graded modules over a PID (see~\ref{prop classification}).


 \section{Noise}\label{noise}
An important step in extracting topological features from a data set is to ignore noise.
Depending on the situation, noise can mean different things. In this section we discuss what we mean by noise for vector space valued tame functors.   Our    objective  is to be able to mark some functors as small.
Thus for any non-negative
 rational number $\epsilon$, we should have a collection $\cS_{\epsilon}$ of tame functors
 which  we consider to be $\epsilon$-small. This collection is called the  $\epsilon$-{\bf component} of a noise system and its members are called  noise of size at most $\epsilon$.
 Noise systems should satisfy certain natural constrains. Here is a formal definition:
  
\begin{Def}\label{def noise} 
A noise system in $\text{\rm Tame}(\QQ^r,\text{\rm Vect}_K)$ is a collection $\{\cS_{\epsilon}\}_{\epsilon\in \QQ}$ of sets
of tame functors, indexed by  rational non-negative numbers $\epsilon$, such that:
\begin{itemize}
\item the zero functor belongs to $\cS_{\epsilon}$ for any $\epsilon$;
\item if $0\leq \tau<\epsilon$, then $\cS_{\tau}\subseteq \cS_{\epsilon}$;
\item if $0\to F\to G\to H\to 0$ is an exact sequence in $\text{\rm Tame}(\QQ^r,\text{\rm Vect}_K)$, then
\begin{itemize}
\item if $G$ is in $\cS_{\epsilon}$, then so are $F$ and $H$;
\item if 
$F$ is in $\cS_{\epsilon}$ and $H$  is in $\cS_{\tau}$, then $G$  is in $\cS_{\epsilon+\tau}$.
\end{itemize}
\end{itemize}
 \end{Def}
 
 The last requirement  for a noise system is called additivity.
 
Let $\{\cS_\epsilon\}_{\epsilon \in \QQ}$ and $\{\cT_\epsilon\}_{\epsilon \in \QQ}$ be noise systems. If, for any $\epsilon$   in   $\QQ$,   $\cS_\epsilon\subset \cT_\epsilon$, then we write $\{\cS_\epsilon\}_{\epsilon\in\QQ}\leq
\{\cT_\epsilon\}_{\epsilon\in\QQ}$. With this relation, noise systems in $\text{\rm Tame}(\QQ^r,\text{\rm Vect}_K)$ form a  poset. This poset has the unique  minimal element given by the sequence  whose components contain only the zero functor. It has also the unique maximal element given by the  sequence whose components contain all tame functors.   Note that  the intersection of any family of noise systems is again a noise system.  This implies for example that the poset of noise systems is a lattice. Moreover, for any sequence of sets $\{\cS_\epsilon\}_{\epsilon\in \QQ}$ of tame functors, the intersection of all the noise systems $ \{\cT_\epsilon\}_{\epsilon\in\QQ}$ for which $\cS_\epsilon \subset \cT_\epsilon$, for any $\epsilon$ in $\QQ$, is the smallest noise system containing  $S_\epsilon$ in its $\epsilon$-component. We call it  the noise system {\bf generated by the sequence}   $\{\cS_\epsilon\}_{\epsilon\in \QQ}$ and denote it by $\langle\{\cS_\epsilon\}_{\epsilon\in \QQ}\rangle$. 

Let  $\{\cS_\epsilon\}_{\epsilon\in\QQ}$ be a  noise system. Define $\cS_{\epsilon}^{c}:=\{F\in \cS_{\epsilon}\ |\ 
F\text{ is compact}\}$. One can use Proposition~\ref{prop basickan}(10) to see that $\{\cS^c_\epsilon\}_{\epsilon\in\QQ}$  is also a noise system.
We call it the compact part of $\{\cS_\epsilon\}_{\epsilon\in\QQ}$.
It follows that if $\{\cS_\epsilon\}_{\epsilon\in \QQ}$ is a sequence of sets of compact tame functors, then $\langle\{\cS_\epsilon\}_{\epsilon\in \QQ}\rangle$ consists of compact functors.

By definition  the $0$-component $\cS_0$ of any noise system  $\{\cS_\epsilon\}_{\epsilon\in\QQ}$ is a  Serre subcategory of  $\text{\rm Tame}(\QQ^r,\text{\rm Vect}_K)$ (see~\cite{Serre2}). In particular the direct sum of  two  functors in $\cS_0$ is again  a functor in  $\cS_0$. 
Since this is not true in general for other components, we need to introduce
a definition: 
a component $\cS_{\epsilon}$ of a noise system is said to be {\bf closed under direct sums} if, for any $F$ and $G$ in $\cS_{\epsilon}$, the direct sum
$F\oplus G$ also belongs to  $\cS_{\epsilon}$. Being closed under direct sums is  important for some of our constructions as in this case,
for any $\epsilon>0$,
any compact and tame functor has a unique maximal subfunctor that belongs to $\cS_{\epsilon}$ (see Proposition \ref{maximal noise}).
 In Section~\ref{sec directsums} we  try to understand under what circumstances a noise system is closed under direct sums.

We now present several examples of noise systems. The first two generalise what we  interpret as noise in the context of persistent homology induced by interleaving distance (see \cite{persistence}).

\begin{point}\label{pt stannoisecone}{\bf Standard Noise in the direction of a cone.} 
Let  $V\subset \QQ^r$ be a subset. Set:
\[V_{\epsilon}:=\left\{F\in\text{\rm Tame}(\QQ^r,\text{\rm Vect}_K) \ \middle|\ {\begin{tabular}{c} \text{for any $u$ in $\QQ^r$ and for any   $x$ in $F(u)$,}\\ \text{there is $w$ in $V$ such that }\\
\text{$||w||=\epsilon$ and $x$ is in $\text{ker}\left(F(u\leq u+w)\right)$}\end{tabular}} \right\}\]
We claim that if $V$ is a cone (see~\ref{pt cone}), then the sequence $\{V_\epsilon\}_{\epsilon\in \QQ}$ is a noise system which we
call  the {\bf standard noise in the direction of the cone $V$}.
It is clear that the zero functor belongs to $V_{\epsilon}$ for any $\epsilon$. Let  $0<\tau<\epsilon$.
If $x$ is in $\text{ker}(F(u\leq u+w))$, then $x$ is also in $\text{ker}(F(u\leq u+\frac{\epsilon}{\tau} w))$, since $w\leq \frac{\epsilon}{\tau} w$.
As  $||\frac{\epsilon}{\tau} w||=\frac{\epsilon}{\tau} ||w|| $,
the inclusion $V_{\tau}\subset V_{\epsilon}$ follows. Consider now an exact sequence $0\to F\to G\to H\to 0$
of tame functors.
If  $G$ is in $V_{\epsilon}$, then, by naturality of $F\hookrightarrow  G$ and $G\twoheadrightarrow H$,
 both functors $F$ and $H$ are also in $V_{\epsilon}$. 
   Assume $F$ is in $V_\epsilon$ and $H$ is in $V_{\tau}$.
 Take an element $x\in G(u)$. Its image $x_1$ in $H(u)$ is therefore in $\text{ker}(H(u\leq u+w))$ for some 
 $w$ in $V$ with $||w||=\tau$. This means that $G(u\leq u+w)$ takes $x$ to an element $x_2$ in $F(u+w)\subset G(u+w)$. We can thus find $w'$ in $V$ with $||w'||=\epsilon$ such that $x_2$ is in $\text{ker}(F(u+w\leq u+w+w'))$.
 It follows that $x$ is in  $\text{ker}(G(u\leq u+w+w'))$.  Since $||w+w'||\leq ||w||+||w'||=\tau+\epsilon$, $x$ is therefore also in $\text{ker}(G(u\leq u+\frac{\tau+\epsilon}{||w+w'||}(w+w')))$. 
 The assumption that $V$ is a cone guarantees that $\frac{\tau+\epsilon}{||w+w'||}(w+w')$ belongs to $V$.
We  can conclude  $G$ belongs to $V_{\epsilon+\tau}$.


Note that the hypothesis that $V\subseteq \QQ^r$ is a cone is fundamental  for  $\{V_\epsilon\}_{\epsilon\in \QQ}$ to be a noise system. To illustrate this consider for example $r=2$,  $V$ to be the set union of two axes $\{(a,0)\ |\ a\in \QQ\}\cup \{(0,b)\ |\ b\in \QQ\}$, and $u=(0,0)$.  The tame functors $F,H\colon \QQ^r \to \text{\rm Vect}_K$ given respectively by:
\[F(v_1,v_2)=\begin{cases}
 K &\text{ if } v_2 <1\\
0 &\text{ otherwise }
\end{cases} \quad \text{and} \quad
H(v_1,v_2)=\begin{cases}
 K &\text{ if } v_1 <1\\
0 &\text{ otherwise }
\end{cases}
\]
are in $V_{1}$.
The functor $G:\QQ^r \to \text{\rm Vect}_K$ defined as:
\[G(v_1,v_2)=\begin{cases}
 K &\text{ if } v_1 <1\ \text{or} \ v_2 <1  \\
0 &\text{ otherwise }
\end{cases}
\]
fits into an exact sequence 
$0\to F\to G\to H\to 0$ but is not in $V_{\epsilon}$ for any positive rational number $\epsilon$.  

In general,  neither $V_{\epsilon}$ nor its compact part  $V_{\epsilon}^c$ are closed under direct sums. For example consider $w=(1,0,1)$ and
 $w'=  (1/2,1,0)$ in $\QQ^ 3$.
Define  $F\colon \QQ^3 \to \text{\rm Vect}_K$ to be the tame functor such that  $F(v)=0$ if $v\geq w$ and $F(v)=K$
otherwise, with $F(u\leq v)$ being the identity  if $F(u)$ and $F(v)$ are  non zero. Similarly, let $G\colon \QQ^3 \to \text{\rm Vect}_K$ be 
a tame functor such that  $G(v)=0$ if $v\geq w'$ and $G(v)=K$ otherwise, with $G(u\leq v)$ being the identity  if $G(u)$ and $G(v)$ are  non zero. Note that $F$  is  in $\text{Cone}(w)_{1}^{c}$ and $G$  is  in $\text{Cone}(w')_{1}^{c}$. 
Although $F$ and $G$ are both in $\text{Cone}(w,w')^c_{1}$, the functor $F\oplus G$ is not since there is no vector $z$ in $\text{Cone}(w,w')$ such that $||z||=1$, $z \geq w$ and $z \geq w'$.   


Note that in the case $r=1$, the $\epsilon$-component of the standard noise in the direction of the ray $\QQ$ coincides with the set of $\epsilon$-trivial persistence modules as defined in \cite{BauerLesnick}.

\end{point}


\begin{point}{\bf The compact part of the standard noise in the direction of a ray.}\label{pt sandal}
Let $V$ be a ray (a cone generated by one element, see~\ref{pt cone}). Then there is a unique $w$ in $V$ such that
$||w||=1$. In this case $F$ belongs to $V_{\epsilon}$ if and only if $F(v)=\text{ker}(F(v\leq v+\epsilon w))$, for any $v$ in $\QQ^r$ (the map $F(v\leq v+\epsilon w)$ is the zero map).  For example 
the cokernel of the unique inclusion $K(v+ \epsilon w,-)\subset K(v,-)$  belongs to $V_{\epsilon}$. 
Recall that   this cokernel is denoted by $[v, v+\epsilon w)$ (see~\ref{ex bar}).  Note that this cokernel is compact and hence it belongs to $V_\epsilon^c$.  Furthermore any finite direct sum $\oplus _{i=1}^n[v_i, v_i+\epsilon w)$ is also  a member of 
$V_\epsilon^c$. We claim that 
$\{V_\epsilon^c\}_{\epsilon\in \QQ}$
is the smallest noise system containing all  such finite direct sums  in its $\epsilon $-component for any $\epsilon$.
In other words  $\{V_\epsilon^c\}_{\epsilon\in \QQ}$ is the noise system generated by a sequence of sets $\{S_{\epsilon}\}_{\epsilon\in \QQ}$, where $S_{\epsilon}$ is the set of all functors of the form
$\oplus_{i=1}^n[v_i,v_i+\epsilon w)$. We have just explained the relation  $\langle\{S_{\epsilon}\}_{\epsilon\in \QQ}\rangle \leq
 \{V_\epsilon^c\}_{\epsilon\in \QQ}$. Let $F$ be in $V^c_\epsilon$. Recall that any element $x$ in $F(v)$
induces  a unique natural transformation $x\colon K(v,-)\to F$ (see~\ref{point minimal}). Since its precomposition
with $K(v+\epsilon w,-)\subset K(v,-)$  is trivial, $x\colon K(v,-)\to F$ factors as $K(v,-)\to [v,v+\epsilon w)\to F$.
This, together with compactness,   implies that 
$F$ is a quotient of a  finite direct sum of functors of the form $ [v,v+\epsilon w)$ which  implies  $F$ is in the $\epsilon$-component of $\langle\{S_{\epsilon}\}\rangle_{\epsilon\in \QQ}$.

Since a direct sum of zero maps is a zero map, the collections $V_{\epsilon}$ and $V_{\epsilon}^c$ are preserved by direct sums for any $\epsilon$.
\end{point}
\begin{point}{\bf Standard Noise in the direction of a sequence of vectors.} \label{pt stnoisedirvectors}
Let us choose a finite sequence $\V=\{v_1,v_2\ldots, v_n\}$ of elements in $\QQ^r$.
For any $w$ in $\text{Cone}(\V)=\text{Cone}(v_1,v_2\ldots, v_n)$, consider the set $T(w)$ of sequences $(a_1,\ldots, a_n)$ of non-negative rational numbers such that $w=a_1v_1+\cdots+a_nv_n$. Define the {\bf $\V$-norm}
as:
\[||w||_{\V}=\text{min}_{\{a_1,\ldots, a_n\}\in T(w)}||(a_1,\ldots, a_n)||=\text{min}_{\{a_1,\ldots, a_n\}\in T(w)}\text{max}_{1\leq i\leq n}a_i\]
Set:
\[\V_{\epsilon}:=\left\{F\in\text{\rm Tame}(\QQ^r,\text{\rm Vect}_K)\ \middle|\
\begin{tabular}{c} \text{for any  $v$ in $\QQ^r$ and for any $x$ in $F(v)$,}\\ \text{ there is $w$ in $\text{Cone}(\V)$ s.t. $||w||_{\V}=\epsilon$ }\\
\text{and $x$ is in $\text{ker}\left(F(v\leq v+w)\right)$}\end{tabular}
\right\}
\]
One can check that $||aw||_{\V}=a||w|||_{\V}$ and $||u+w|||_{\V}\leq ||u|||_{\V}+||w|||_{\V}$ for any 
$v$ and $w$   in $\text{Cone}(\V)$ and any $a$   in $\QQ$. Exactly the same arguments as in~\ref{pt stannoisecone}
 can be then used to prove that   $\{\V_{\epsilon}\}_{\epsilon\in \QQ}$ is also a noise system. We call it the {\bf standard noise in the direction of the sequence $\V$}.
 
 For example let $v$ in $\QQ^r$ be non-zero and $\V=\{v\}$.  In this case for any $w$   in  $\text{Cone}(v)$, $||w||_{\V}=||w||/||v||$ and   $\V_{\epsilon}=\text{Cone}(v)_{\epsilon/||v||}$ for any $\epsilon$ in $\QQ$.
 \end{point}
\begin{point}{\bf Domain noise.}\label{domain noise}
Let   $\calx=\{X_{\epsilon}\}_{\epsilon\in \QQ}$ be a sequence of subsets of $\QQ^r$ with the property that if $0\leq \tau<\epsilon$, then $X_{\tau} \subseteq X_{\epsilon}$. 
For a tame functor $F:\QQ^r\to \text{\rm Vect}_K$ we define 
  $\text{domain}(F):=\{v \in \QQ^r \ | \ F(v) \ne 0\}$ and call it the {\bf domain} of $F$. For example the domain of the zero functor is  empty.
Set:
\[
\calx_{\epsilon}:=\{F\in\text{\rm Tame}(\QQ^r,\text{\rm Vect}_K)\ |\
\text{domain}(F) \subseteq X_{\epsilon} \}.
\]
The fact that $\{\calx_{\epsilon}\}_{\epsilon\in \QQ}$ is a noise system is a direct consequence of the fact that $X_{\tau} \subseteq X_{\epsilon}$ for any $0\leq\tau<\epsilon$.
This noise system satisfies an extra condition. For any exact sequence $0\to F\to G\to H\to 0$ of tame functors,
if $F$ is in $ \calx_{\epsilon}$ and $H$   in $\calx_{\tau}$, then $G$ is in $\calx_{\text{max}\{\epsilon,\tau\}}$.
This implies in particular that both  $\calx_{\epsilon}$  and  $\calx_{\epsilon}^c$  are closed under direct sums.\end{point}
\begin{point}{\bf Dimension noise.}\label{dimensionnoise} 
Let $\caln=\{n_{\epsilon}\}_{\epsilon\in \QQ}$ be a sequence of natural numbers such that  $n_0=0$ and 
 $n_{\tau}+n_{\epsilon}\leq n_{\tau+\epsilon}$
for any $\tau$ and $\epsilon$ in $\QQ$. 
Set:
\[
\caln_{\epsilon}:=\{F\in\text{\rm Tame}(\QQ^r,\text{\rm Vect}_K)\ |\ \text{for any $v$ in $\QQ^r$, }  \text{dim}_K F(v)\leq n_\epsilon \}.
\]
The proof that  $\{\caln_{\epsilon}\}_{\epsilon\in \QQ}$ is a noise system is straightforward and only depends on the facts that  $n_0=0$, 
 $\{n_{\epsilon}\}_{\epsilon\in \QQ}$ is a non decreasing sequence of non negative numbers, and that an sequence of tame functors $0\to F\to G\to H\to 0$ is exact if it is object wise exact.
%
\end{point}

More examples of noise systems can be produced using the property that the intersection of an arbitrary family of noise systems is a noise system.  For example:
\begin{point}{\bf Intersection noise.}
Let $L^1, \ldots ,  L^n$ be rays in $\QQ^r$. Choose the unique $w_i$ in $L^i$ such that $||w_i||=1$.
The intersection of $\{L^{i}_{\epsilon}\}_{\epsilon\in \QQ}$ for $1\leq i\leq n$ is the noise system whose  $\epsilon$-component consists of these tame functors $F:\QQ^r \rightarrow {\rm Vect}_K$
for which $F(v\leq v+ \epsilon  w_i)$  is the zero map for any $v$ and any $1\leq i\leq n$.
\end{point}

\begin{point}{\bf Noise generated by a functor.}
Let $M$ be a tame functor and $\alpha$ a positive rational number. 
Let $\langle M,\alpha \rangle$ be the smallest noise system such that $M$ is in $\langle M,\alpha \rangle_{\alpha}$.
This might be interesting in cases in which one wants to declare some functor as noise of a certain size.
The collection $\langle M,\alpha \rangle_{\epsilon}$ can be described inductively as follows:
 \[\langle M,\alpha \rangle_{\epsilon}=\begin{cases}
\{0\} &\text{if } \ 0\leq \epsilon<\alpha  \\
[M] &\text{if} \ \epsilon=\alpha\\
\left[\cup_{i=1} ^{n-1} \langle M,\alpha\rangle_{i\alpha}\ \cup\ \text{Ext}(\langle M,\alpha\rangle_{(n-i)\alpha})\right] & \text{if} \ \epsilon=n \alpha \ \text{for}\ n>1\\
\langle M,\alpha\rangle_{n\alpha} & \text{if } n\alpha\leq \epsilon<(n+1)\alpha
\end{cases}
\]
where  $[K]$ denotes  the set of all tame subfunctors and quotients of $K$ and 
 $\text{Ext}(K)$ denotes  the collection of all tame functors $G$ that fit into an exact sequence of the form
$0\to F\to G\to H\to 0$ where $F$   and $H$ are in $[K]$.
\end{point}

\section{Noise systems closed under direct sums}\label{sec directsums}
Let $\{\cS_{\epsilon}\}_{\epsilon\in \QQ}$ be a noise system in $\text{\rm Tame}(\QQ^r,\text{\rm Vect}_K)$ and
 $F\colon\QQ^r\to\text{Vect}_K$ be a tame and compact functor. 
Consider the collection of all subfunctors of $F$  that belong to $\cS_{\epsilon}$.  Because of the compactness of $F$, Kuratowski-Zorn lemma implies that this collection has maximal elements with respect to the inclusion. In general however there could be  many  such maximal elements.
In this section we discuss under what circumstances there is only one maximal element in this collection. The subfunctor corresponding to this element   is  the unique  maximal noise of size $\epsilon$ inside $F$ and we will use it to denoise   $F$. If it exists, we denote this maximal subfunctor by $F[\cS_{\epsilon}]\subset F$. By definition
the inclusion $F[\cS_{\epsilon}]\subset F$ satisfies the following universal property: $F[\cS_{\epsilon}]$ belongs to $\cS_{\epsilon}$ and, for any $G$ in  $\cS_{\epsilon}$,
 any natural transformation $G\to F$  maps  $G$ into $F[\cS_{\epsilon}]\subset F$.  Thus for any $G$ in  $\cS_{\epsilon}$, the inclusion $F[\cS_{\epsilon}]\subset F$ induces a bijection between $\text{Nat}(G,F)$ and $\text{Nat}(G,F[\cS_{\epsilon}])$. 
   
\begin{prop}\label{maximal noise}
Let $\{\cS_{\epsilon}\}_{\epsilon\in \QQ}$ be a noise system in $\text{\rm Tame}(\QQ^r,\text{\rm Vect}_K)$.
The component  $\cS_{\epsilon}^c$ is closed under  direct sums if and only if  $F[\cS_{\epsilon}]\subset F$ exists
for any tame and compact functor  $F\colon\QQ^r\to\text{\rm Vect}_K$.
\end{prop}
\begin{proof}
Assume  $\cS_{\epsilon}^c$ is closed under  direct sums.  Let   $F\colon\QQ^r\to\text{\rm Vect}_K$ be a tame and compact functor and 
 $G\subset F$ and $H\subset F$ be maximal, with respect to the inclusion, subfunctors such that
$G$ and $H$ are in $\cS_{\epsilon}^c$. Since   $G\oplus H$ is in  $\cS_{\epsilon}^c$, then so is its quotient
$G+H\subset F$. Using maximality of $G\subset F$ and $H\subset F$, we obtain equalities $G=G+H=H$.
We can conclude that there is a unique maximal subfunctor of $F$ that belongs to $\cS_{\epsilon}$.

Assume now that, for any tame and compact functor  $F\colon\QQ^r\to\text{\rm Vect}_K$,  there is a unique maximal  $F[\cS_{\epsilon}]\subset F$  such that $F[\cS_{\epsilon}]$ belongs to  $\cS_{\epsilon}$.
Let $G$ and $H$    be in  $\cS_{\epsilon}^c$. Consider $(G\oplus H)[\cS_{\epsilon}]$. By the maximality,
we have inclusions $G\subset (G\oplus H)[\cS_{\epsilon}]\supset H$, and thus  $G\oplus H=(G\oplus H)[\cS_{\epsilon}]$. The functor   $G\oplus H$ is therefore in $\cS_{\epsilon}^c$.
\end{proof}

\begin{cor}
Let $\{\cS_{\epsilon}\}_{\epsilon\in \QQ}$ be a noise system in $\text{\rm Tame}(\QQ^r,\text{\rm Vect}_K)$. Assume that  
the component  $\cS_{\epsilon}^c$ is closed under  direct sums. Then, for any tame and compact functors
$F,G\colon\QQ^r\to \text{\rm Vect}_K$, $F[\cS_\epsilon]\oplus G[\cS_\epsilon]=(F\oplus G)[\cS_\epsilon]$.
\end{cor}
\begin{proof}
For any $H$ in $\cS_\epsilon$, we have a sequence of bijections induced by the appropriate inclusions:
\[\text{Nat}(H,(F\oplus G)[\cS_\epsilon])=\text{Nat}(H,F\oplus G)=\text{Nat}(H,F)\oplus \text{Nat}(H,G)=\]
\[=
\text{Nat}(H,F[\cS_\epsilon])\oplus \text{Nat}(H,G[\cS_\epsilon])=\text{Nat}(H,F[\cS_\epsilon]\oplus G[\cS_\epsilon])\]
This shows   $F[\cS_\epsilon]\oplus G[\cS_\epsilon]\subset (F\oplus G)[\cS_\epsilon]$
is an isomorphism.
\end{proof}

Based on the above proposition, we are going to look for noise systems whose  compact parts are closed under direct sums. The key example of such  a noise system  is the standard noise in a direction of a ray (see~\ref{pt sandal})
or a vector (see~\ref{pt stnoisedirvectors}).
More generally:

\begin{prop}\label{prop directsumstandnoisecone}
Let $V$ be a cone in $\QQ^r$ and $\{V_{\epsilon}\}_{\epsilon\in \QQ}$ be the standard noise in the direction of the cone $V$.  The following are equivalent:
\begin{enumerate}
\item The  collection $V_{\epsilon}$ (see~\ref{pt stannoisecone}) is closed under direct sums.
\item The  collection $V_{\epsilon}^c$ (see~\ref{pt stannoisecone}) is closed under direct sums.
\item For any  $w_1$ and  $w_2$  in the cone $V$ whose norm is $\epsilon$ ($||w_1||=||w_2||=\epsilon$), there is an element $w$ in $V$ of norm $\epsilon$ ($||w||=\epsilon$) such that $w_1\leq w$, $w_2\leq w$.
\end{enumerate}
\end{prop}
\begin{proof}
The implication (1)$\Rightarrow$(2) is clear. 
 Assume (2). We show (3). Let     $w_1$ and  $w_2$ be elements of   the cone $V$ such that $||w_1||=||w_2||=\epsilon$. Consider the functors  $[0,w_1)$ and $[0,w_2)$ (see~\ref{ex bar}).
Since they belong to $V_{\epsilon}^c$, then, by the assumption, so does their direct sum  $F:=[0,w_1)\oplus [0,w_2)$. Consider $x$ in
$F(0)=([0,w_1)\oplus [0,w_2))(0)=[0,w_1)(0)\oplus [0,w_2](0)=K\oplus K$ given by the diagonal element $(1,1)$.
As $F$ is in $V_{\epsilon}$, there is $w$ in $V$  such that $||w||=\epsilon$ and $x$ is in the kernel of 
$F(0\leq w)$.  This can happen only if $w_1\leq w$ and $w_2\leq w$. Thus $w$  is the desired element.


Assume (3).  We prove (1). Let $F$ and $G$ be in $V_{\epsilon}$. We need to show that $F\oplus G$ also belongs to $V_{\epsilon}$. Choose $(x,y)$ in $F(v)\oplus G(v)$. Let $w_x$ and $w_y$ be two elements in the cone $V$ of norm $\epsilon$ such that $x$ is in $\text{ker}(F(v\leq v+w_x))$ and $y$ is in  $\text{ker}(G(v\leq v+w_y))$.
By the assumption there is $w$ in $V$ of norm $\epsilon$  with $w_x\leq w$ and $w_y\leq w$.
Thus $x$ is in $\text{ker}(F(v\leq v+w))$ and $y$ in $\text{ker}(G(v\leq v+w))$. It follows that
$(x,y)$ is in $\text{ker}((F\oplus G)(v\leq v+w))$. As this happens for any $(x,y)$, the direct sum $F\oplus G$ belongs to $V_{\epsilon}$.
 \end{proof}
 
 Exactly the same argument as in the proof of~\ref{prop directsumstandnoisecone}, can be also applied to show
 an analogous statement   for the standard noise in the direction of a sequence of vectors in $\QQ^r$:
 
 \begin{prop}\label{sequence of vectors} Let $\V=\{v_1,\ldots, v_n\}$ be a sequence of vectors in $\QQ^r$
 and  $\{\V_{\epsilon}\}_{\epsilon\in \QQ}$ be the standard noise in the direction of the sequence $\V$.
 The following are equivalent:
 \begin{enumerate}
 \item  The  collection $\V_{\epsilon}$ (see~\ref{pt stnoisedirvectors}) is closed under direct sums.
\item The  collection $\V_{\epsilon}^c$ (see~\ref{pt stnoisedirvectors}) is closed under direct sums.
\item For any  $w_1$ and  $w_2$  in  $\text{\rm Cone}(\V)$ whose $\V$-norm is $\epsilon$ ($||w_1||_{\V}=||w_2||_{\V}=\epsilon$), there is an element $w$ in $V$ of $\V$-norm $\epsilon$ ($||w||_{\V}=\epsilon$) such that $w_1\leq w$, $w_2\leq w$.
\end{enumerate}
 \end{prop}
 
We finish this section with:
\begin{prop}\label{prop sup} Let $\V=\{v_1,\ldots, v_n\}$ be a sequence of elements in $\QQ^r$.
\begin{enumerate}
\item   Let  $V=\text{\rm Cone}(v_1,\ldots, v_n)$ and  $L=\text{\rm Cone}(v_1+\cdots +v_n)$. If 
$||v_1||=\cdots=||v_n||=||v_1+\cdots+v_n||$, then $\{V_\epsilon\}_{\epsilon\in\QQ}=\{L_\epsilon\}_{\epsilon\in\QQ}$.
\item Let  $\W=\{v_1+\cdots +v_n\}$. If $v_1,\ldots, v_n$  are linearly independent as vectors over the field of rational numbers, then $\{\V_\epsilon\}_{\epsilon\in\QQ}=\{\W_\epsilon\}_{\epsilon\in\QQ}$.
\end{enumerate}
\end{prop}

\begin{proof}
\noindent
(1):\quad
Set $w:=(\epsilon/||v_1|| )(v_1+\cdots+v_n)$. By the assumption $||w||=\epsilon$. Let $u$ in $V$ be of norm $\epsilon$. Thus  $u$ can be written as $u=a_1 v_1+ \ldots + a_n v_n$ where $0\leq a_i \leq \epsilon/||v_1||$ for every $1\leq i\leq n$. It follows that $u\leq (\epsilon/||v_1|| )(v_1+\cdots+v_n)=w$.

Let $F$ be in $V_{\epsilon}$. This means that, for any $x$ in $F(v)$, there  exists $w_x$ in $V$ of norm $\epsilon$ such that $x$ is in the kernel of $F(v\leq v+ w_x).$  We have already shown that $w_x\leq w$.
It thus follows that $x$   is also in the kernel of $F(v\leq v+ w).$ As this is true for any $x$, 
$F(v\leq v+ w)$ is the zero map. This means $V_{\epsilon}=\text{Cone}(w)_{\epsilon}$
\medskip

\noindent
(2):\quad
Set $w=v_1+\cdots+v_n$. Since $ \{v_1,\ldots, v_n\}$ are linearly independent, $||\epsilon w||_{\V}=\epsilon$.
Let $u$  be in $\text{\rm Cone}(v_1,\ldots, v_n)$ of $\V$-norm $\epsilon$.  
Thus $u$ÃÂ  can be written as $u=a_1 v_1+ \ldots + a_n v_n$ where $0\leq a_i \leq \epsilon$.
It follows that $u\leq \epsilon (v_1+\cdots +v_n)=\epsilon w$. 

Let $F$ be in $\V_{\epsilon}$. This means that, for any $x$ in $F(v)$, there  exists $w_x$ in $V$ of $\V$-norm $\epsilon$ such that $x$ is in the kernel of $F(v\leq v+ w_x).$  We have already shown that $w_x\leq \epsilon w$.
It thus follows that $x$ is also in the kernel of $F(v\leq v+ \epsilon w).$ As this is true for any $x$, 
$F(v\leq v+ \epsilon w)$ is the zero map. This means $V_{\epsilon}=\W_{\epsilon}$.
 \end{proof}
 
The special case of~\ref{prop sup} we are most interested in is when the vectors  $v_1,\ldots,v_n$  are among the standard  vectors of $\QQ^r$ (see~\ref{point NRposets}).

\section{Topology on Tame functors}\label{sec topologytame}
In this section we describe how a noise system leads to a pseudo metric and hence a topology on the set of 
 tame functors with values in $ \text{\rm Vect}_K$. This metric can be used  to measure 
 how close or how far apart  tame  functors can be relative to the chosen noise system.
 Let us choose and fix a noise system $\{\cS_{\epsilon}\}_{\epsilon\in \QQ}$ in  $\text{Tame}(\QQ^r, \text{\rm Vect}_K)$.

We are going to compare tame functors  using natural transformations. Let
$\phi\colon F\to G$ be a natural transformation between functors in $\text{Tame}(\QQ^r, \text{\rm Vect}_K)$. 
We say that $\phi$ is an {\bf $\epsilon$-equivalence} if, there are $\tau$ and $\mu$  in $\QQ$ such that
$\tau+\mu\leq \epsilon$, $\text{\rm ker}(\phi)$ belongs to $\cS_{\tau}$ and $\text{\rm coker}(\phi)$ belongs to $\cS_{\mu}$. 
Before we define a pseudometric and a topology on the set of  tame functors, we need to  prove   two fundamental properties of
being an $\epsilon$-equivalence. The first one is the preservation of $\epsilon$-equivalences by  both push-outs and pull-backs:

\begin{prop}\label{prop pushpulleequiv}
Consider the following commutative square in  $\text{\rm Tame}(\QQ^r, \text{\rm Vect}_K)$:
\[\xymatrix@R=12pt@C=12pt{
 & H\dlto_-{\phi} \drto\\
 F\drto & & G\dlto^-{\phi'}\\
 & P
}\]
\begin{enumerate}
\item Assume that the square is a push-out. If $\phi$ is an $\epsilon$-equivalence, the same holds for $\phi'$.
\item  Assume that the square is a pull-back. If $\phi'$ is an $\epsilon$-equivalence, the same holds for $\phi$.
\end{enumerate}
\end{prop}
\begin{proof}
As the proofs of (1) and (2) are analogous, we present only a sketch of the argument for (1).  
We claim that the assumption that the square is a push-out implies that  $\text{\rm coker}(\phi)$ is isomorphic to $\text{\rm coker}(\phi^{\prime})$ and $\text{\rm ker}(\phi^{\prime})$ is a quotient of 
$\text{\rm ker}(\phi)$.
The proposition then clearly follows as a component of a noise system is closed under quotients.
As push-outs in the  functor category are formed object-wise, it is then enough to show the claim  in the category of
vector spaces
$\text{\rm Vect}_K$. This is left as an easy exercise. 
\end{proof}

The second  property of $\epsilon$-equivalences is additivity with respect to the scale $\epsilon$:
\begin{prop}\label{prop additiveequiv}
Let $\phi\colon F\to G$ and $\psi\colon G\to H$ be natural  transformations between functors in $\text{\rm Tame}(\QQ^r, \text{\rm Vect}_K)$.
 If $\phi$ is an $\epsilon_1$-equivalence, and $\psi$   is an $\epsilon_2$-equivalence, then $\psi\phi$ is 
 an $(\epsilon_1+\epsilon_2)$-equivalence.
\end{prop}
\begin{proof}
The key is to observe that the natural transformations $\phi$ and $\psi$ induce the following exact sequences:
\[
\xymatrix{
0\rto &\text{ker}(\phi)\rto & \text{ker}(\psi\phi)\rto & \text{ker}(\psi)
}
\]
\[
\xymatrix{
&\text{coker}(\phi)\rto & \text{coker}(\psi\phi)\rto & \text{coker}(\psi)\rto & 0
}
\]
Let $\tau_1$ and $\mu_1$ in $\QQ$ be such that  $\tau_1+\mu_1\leq \epsilon_1$ and 
$\text{ker}(\phi)$ belongs to $\cS_{\tau_1}$ and $\text{coker}(\phi)$ belongs to $\cS_{\mu_1}$.
Similarly let   $\tau_2$ and $\mu_2$ in $\QQ$ be such that  $\tau_2+\mu_2\leq \epsilon_2$ and 
$\text{ker}(\psi)$ belongs to $\cS_{\tau_2}$ and $\text{coker}(\psi)$ belongs to $\cS_{\mu_2}$.
Therefore, the image of $\text{ker}(\psi\phi)\to  \text{ker}(\psi)$, as a subfunctor in $\text{ker}(\psi)$, belongs to $\cS_{\tau_2}$ and the kernel of $\text{coker}(\psi\phi)\to \text{coker}(\psi)$, as a  quotient of $\text{coker}(\phi)$, belongs to 
$\cS_{\mu_1}$. We can then conclude that $\text{ker}(\psi\phi)$ belongs to $\cS_{\tau_1+\tau_2}$ and
$\text{coker}(\psi\phi)$ belongs to  $\cS_{\mu_1+\mu_2}$. Since 
$\tau_1+\tau_2+\mu_1+\mu_2\leq \epsilon_1+\epsilon_2$, the transformation 
 $\psi\phi$  is an
$(\epsilon_1+\epsilon_2)$-equivalence.
\end{proof}

We can use the above fundamental properties of $\epsilon$-equivalences to prove:

\begin{cor} \label{cor equivpushpull}
Let $F$ and $G$ be tame functors and $\tau$ and $\mu$ be  non-negative  rational numbers. Then the following statements are equivalent:
\begin{enumerate}
\item There are natural transformations $ F\leftarrow H:\!\phi$ and $\psi\colon H\to G$  such that
$\phi$ is a $\tau$-equivalence and $\psi$  is a $\mu$-equivalence.
\item There are natural transformations $\psi^\prime \colon F\to P\leftarrow G:\!\phi^\prime$ such that
$\phi^\prime$ is a $\tau$-equivalence and $\psi^\prime$  is a $\mu$-equivalence.
\end{enumerate}
\end{cor}
\begin{proof}
Assume (1) and apply~\ref{prop pushpulleequiv}.(1) to the following push-out square to get (2):
\[\xymatrix@R=12pt@C=12pt{
 & H\dlto_{\phi} \drto^{\psi}\\
 F\drto_-{\psi^\prime} & & G\dlto^-{\phi^\prime}\\
 & P
}\]

The opposite implication (2)$\Rightarrow$(1) follows from a similar argument by applying~\ref{prop pushpulleequiv}.(2)
to an appropriate pull-back square. 
\end{proof}

We are now ready for our key definition:

\begin{Def}\label{def eclose}
Let $\epsilon$ be in $\QQ$. 
Two tame functors $F$ and $G$ are $\epsilon$-{\bf close} if there are natural transformations 
 $ F\leftarrow H:\!\phi$ and $\psi\colon H\to G$  such that
$\phi$ is a $\tau$-equivalence, $\psi$  is a $\mu$-equivalence, and $\tau+\mu\leq \epsilon$. 
\end{Def}

Note that if  there is an $\epsilon$-equivalence $\phi\colon F\to G$ or an $\epsilon$-equivalence $\phi \colon G \to F$ then $F$ and $G$ are $\epsilon$-close.   

According to~\ref{cor equivpushpull}, two tame functors $F$ and $G$ are $\epsilon$-close if and only if
there are natural transformations $\psi^\prime\colon F\to P\leftarrow G:\!\phi^\prime$ such that
$\phi^\prime$ is a $\tau$-equivalence and $\psi^\prime$  is a $\mu$-equivalence and $\tau+\mu \leq \epsilon$.

It is  not true that  any  two tame functors are  $\epsilon$-close for some $\epsilon$ in $\QQ$. For example let $r=1$ and consider the standard noise in the direction of $\QQ$. The free functor
 $K(0,-) \colon \QQ \to \text{\rm Vect}_K$
 is not $\epsilon$-close to the zero functor for any $\epsilon$ in $\QQ$.
We say that two functors are {\bf close} if they are $\epsilon$-close for some $\epsilon$ in $\QQ$.

If $\cS_0$ contains a non-zero functor, then this functor is $0$-close to the zero functor. 
The intersection $\cap_{\epsilon>0}\cS_\epsilon$ consists of the functors which are $\epsilon$-close to the zero functor for any $\epsilon>0$. If this intersection contains only the zero functor, then a natural transformation
is an isomorphism if and only if it is an $\epsilon$-equivalence for any $\epsilon>0$. In particular two functor  are
isomorphic if and only if they are $\epsilon$-close for any  $\epsilon>0$.

Being $\epsilon$-close is a reflexive and symmetric relation. It is however  not transitive in general. Instead  it is additive with respect to the scale $\epsilon.$
\begin{prop}\label{prop triangular}
Let $F,G,H\colon\QQ^r\to\text{\rm Vect}_K$ be tame functors. If $F$ is $\epsilon_1$-close to $G$ and $G$ is $\epsilon_2$-close to $H$, then $F$ is $(\epsilon_1+\epsilon_2)$-close to $H$.
\end{prop}

\begin{proof}
If $F$ is $\epsilon_1$-close to $G$  then we have morphisms  $F\leftarrow H_1:\!\phi_1$ and $\psi_1\colon H_1\to G$  
where  $\phi_1$ is an $\alpha$-equivalence, $\psi_1$  is a $\beta$-equivalence, and $\alpha+\beta\leq \epsilon_1$. Similarly if $G$ is $\epsilon_2$-close to $H$ we have morphisms $G\leftarrow H_2:\!\phi_2$ and $\psi_2\colon H_2\to H$ where $\phi_2$ is a $\gamma$-equivalence, $\psi_2$  is a $\delta$-equivalence, and $\gamma+\delta\leq \epsilon_2$.
Consider the pull-back square:
\[\xymatrix@R=12pt@C=12pt{
& H_3 \ar[dr]^{\psi_3} \ar[dl]_{\phi_3} &  \\
H_1 \ar[dr]_{\psi_1}&  & H_2 \ar[dl]^{\phi_2} \\
&{G}  & 
}\] 
From Proposition {\ref{prop pushpulleequiv}} it follows that $\phi_3$ is a $\gamma$-equivalence and $\psi_3$ is a $\beta$-equivalence.
The natural transformation $F\leftarrow H_3:\! \phi_1\phi_3$ is a $(\gamma+\alpha)$-equivalence and the natural transformation $\psi_2\psi_3\colon H_3\to H $ is a $(\beta+\delta)$-equivalence
by Proposition \ref{prop additiveequiv}. The claim follows from the fact that $\alpha+\beta+\gamma+\delta \leq \epsilon_1+\epsilon_2$. 
\end{proof}

We use the relation of being $\epsilon$-close to define a pseudometric (see~\ref{pt metric}) on tame functors:
\begin{Def}\label{def distance tame}
Let $F,G\colon\QQ^r\to\text{\rm Vect}_K$ be tame functors. If  $F$ and $G$ are  close we define $d(F,G):=
\text{inf}\{\epsilon \in\QQ  \ |\ \text{$F$ and $G$ are $\epsilon$-close}\}$ and  otherwise  $d(F,G):=\infty$.
\end{Def}

\begin{prop}\label{pseudodistance}
The function $d$, defined in \ref{def distance tame}, is an extended pseudometric on the set of tame functors with values in $\text{\rm Vect}_K$.
\end{prop}

\begin{proof}
 The symmetry $d(F,G)=d(G,F)$ follows from the fact that being
$\epsilon$-close is a symmetric relation. Since the zero functor is in $\cS_{\epsilon}$ for any $\epsilon$, then
 $d(F,F)=0$. The triangle inequality  $d(F,H)\leq d(F,G)+d(G,H)$  is a consequence of Proposition~\ref{prop triangular}. 
\end{proof}

Let $\tau>0$ be a positive real number and  $F$ and $G$ be tame functors. By definition, 
 $d(F,G)<\tau$ if and only if
$F$ and $G$ are $\epsilon$-close for some $\epsilon<\tau$. Thus the open ball $B(F,\tau)$, around $F$ with radius $\tau$
(see~\ref{pt metric}),  consists  of all
tame functors which are $\epsilon$-close to $F$ for some $\epsilon<\tau$. These sets  form  a base of the topology
induced by  the pseudometric defined in~\ref{def distance tame}.

For the standard noise in the direction of full cone $\QQ^r$, the pseudometric defined above 
is related to the interleaving pseudometric introduced by M.\ Lesnick in~\cite{Interleaving Multi}, 
 and for the case of persistent homology $(r=1)$ in  \cite{Interleaving}. 
This work in fact  inspired us towards the formulation of our noise systems. For example one can show that if 
two functors are $\epsilon$-interleaved, then they are $6\epsilon$-close and vice versa if two functors are $\epsilon$-close, then they are $\epsilon$-interleaved.



%


\section{Feature Counting Invariant}\label{basic barcode}
%
%
In this section we describe a  pseudometric space of feature counting functions. This space is the range for our invariant
which we   call a feature counting invariant.  This invariant is a continuous function associated to a 
 noise system in $\text{Tame}(\QQ^r,\text{Vect}_{K})$. Its domain is  the space  of tame functors, and its range is the space of feature counting functions. Our aim in this section is to construct this  feature counting  invariant and show that it is $1$-Lipschitz. Throughout this section let us choose and fix a noise system  $\{\cS_{\epsilon}\}_{\epsilon \in \QQ}$ in 
  $\text{Tame}(\QQ^r,\text{Vect}_{K})$.  All the distances and neighbourhoods in  $\text{Tame}(\QQ^r,\text{Vect}_{K})$ are relative to this choice, as defined in Section~\ref{sec topologytame}.

By definition a {\bf feature counting function} is  a functor $f\colon\QQ\to \N^{\rm op}$.  
We write $f_t$ to denote the value of $f$ at $t$ in $\QQ$.  
Thus a feature counting function is simply a non-increasing sequence of natural numbers
 indexed by non-negative rational numbers $\QQ$.  
In particular  a feature counting function has only finitely many values. 
Note that the category $\text{Fun}(\QQ,\N^{\text{op}})$ is a poset, and there is a 
natural transformation between  $f\colon\QQ\to\N^{\text{op}}$ and $g\colon\QQ\to\N^{\text{op}}$ 
if and only if, for any $t$ in $\QQ$, $f_t\geq g_t$.

Let $\epsilon$ be in $\QQ$.  We say that two feature counting functions $f,g\colon\QQ\to\N^{\text{op}}$ are {\bf $\epsilon$-interleaved} if, for any $t$ in $\QQ$, $f_{t}\geq g_{t+\epsilon}$ and $g_{t}\geq f_{t+\epsilon}$ (this definition follows the definition given in \cite{generlized_interleavings}). 

It is not true  that any two feature counting functions are $\epsilon$-interleaved for some $\epsilon$. For example the constant feature counting functions
$0,1\colon \QQ\to\N^{\rm op}$ with values respectively $0$ and $1$, are not $\epsilon$-interleaved for any $\epsilon$.  
Two feature counting functions  $f$ and $g$ are called {\bf interleaved} if they are 
$\epsilon$-interleaved for some $\epsilon$.  Note that   $f$ and $g$ are 
$0$-interleaved if and only if $f=g$.  It is however not true that if $f$ and $g$ are $\epsilon$-interleaved
for any $\epsilon>0$, then $f=g$. For example, let:
\[f_t=\begin{cases} 1& \text{ if } 0\leq t<1\\
0& \text{ if } t\geq 1
\end{cases}\ \ \ \ 
g_t=\begin{cases} 1& \text{ if } 0\leq t\leq 1\\
0& \text{ if } t> 1
\end{cases}\]
Then, although $f\not=g$, the feature counting functions  $f$ and $g$ are $\epsilon$-interleaved
for any $\epsilon>0$.

Being $\epsilon$-interleaved is a reflexive and symmetric relation. However it is not transitive. Instead  it  is  additive 
 with respect to the scale: if $f$ and $g$ are $\tau$-interleaved  while $g$ and  $h$ are
$\mu$-interleaved, then $f$ and $h$ are $(\tau+\mu)$-interleaved.

We use the notion of being interleaved to define a pseudometric on the set of feature counting functions:
\begin{Def}\label{def basic}
Let $f,g\colon\QQ\to\N^{\text{op}}$ be feature counting functions. If $f$ and $g$ are interleaved we define $d(f,g):=\text{inf}\{\epsilon\ |\ \text{$f$ and $g$ are $\epsilon$-interleaved}\}$ and otherwise $d(f,g):=\infty$.
\end{Def}
The discussion before Definition~\ref{def basic} proves:
\begin{prop}\label{dist basic}
The distance $d$, defined in~\ref{def basic}, is an extended pseudometric on the set of feature counting functions.
\end{prop}



For example, let (here $\pi$ denotes the length of the circle of diameter  $1$):
\[f_t=\begin{cases} 1& \text{ if } 0\leq t<\pi\\
0& \text{ if } t> \pi 
\end{cases} 
\]
Then $d(f,0)=\pi$, where $0$ denotes the constant function with value $0$.

Let $\tau>0$ be a positive real number and  $f$ and $g$ be feature counting functions. By definition, 
 $d(f,g)<\tau$ if and only if
$f$ and $g$ are $\epsilon$-interleaved for some $\epsilon<\tau$. Thus the open ball $B(f,\tau)$ around $f$ with radius $\tau$, consists of all
feature counting functions which are $\epsilon$-interleaved with $f$ for some $\epsilon<\tau$.  These sets form  a base of the topology
induced by  the pseudometric defined in~\ref{def basic}.


\medskip

Recall that we have chosen a noise system $\{\cS_{\epsilon}\}_{\epsilon \in \QQ}$ in 
  $\text{Tame}(\QQ^r,\text{Vect}_{K})$. Together with the rank (see \ref{point minimal}) we are going to use this noise system  to associate a feature counting function to a tame and compact functor. To make the association continuous, we  minimize the rank over the neighbourhoods $B(F,t)$ of a given  compact and  tame functor  $F:\QQ^r\rightarrow \text{\rm Vect}_K$. For $t$   in $\QQ$, define:
\[\text{bar}(F)_{t}:=\begin{cases}
\text{rank}(F) &\text{ if }t=0\\
\text{min} \{\text{rank}(G) \,|\, G \in B(F,t)\}&\text{ if }t>0
\end{cases}\]
Since $F$ is tame and compact, $\text{bar}(F)_{t}$ is a natural number. 
Note that $\text{rank}(F)\geq \text{bar}(F)_{t}$ for any $t$. Furthermore,
if $0<t\leq s$, then $B(F,t)\subset B(F,s)$ and hence $\text{bar}(F)_{t}\geq
\text{bar}(F)_{s}$.  Thus the association $t\mapsto \text{bar}(F)_{t}$ defines a functor
$\text{bar}(F)\colon \QQ \to \N^{\text{op}} $ which we call the {\bf feature counting invariant} of $F$ (with respect to  the noise system 
$\{\cS_{\epsilon}\}_{\epsilon \in \QQ}$).


It is important  to be aware that the feature counting invariant of $F$ depends on the choice of a noise system defining its neighbourhoods 
$B(F,t)$. 
The fundamental fact about the feature counting invariant is that it is   a $1$--Lipschitz  function (see \ref{pt metric}):

\begin{prop}\label{prop:lipschitz} Let 
$F,G:\QQ^r \rightarrow \text{\rm Vect}_K$ be tame and compact. Then:
\[d(\text{\rm bar}(F),\text{\rm bar}(G)) \leq d(F,G)\]
\end{prop}

\begin{proof}
It is enough to show that, for any  $\epsilon$ in $\QQ $, if $F$ and $G$ are $\epsilon$-close then the corresponding feature counting invariants  $\text{bar}(F)$ and $\text{bar}(G)$ are $\epsilon$-interleaved.  
If $F$ and $G$ are $\epsilon$-close, by Proposition \ref{prop triangular}, any functor in $B(F,t)$ is $(t_0+\epsilon)$-close to $G$ for some $t_0 < t$.
This implies that $B(F,t)\subset B(G,t+\epsilon)$ and therefore $\text{bar}(F)_{t}\geq \text{bar}(G)_{t+\epsilon}$.
In the same way, if $F$ and $G$ are $\epsilon$-close, any functor in $B(G,t)$ is $(t_0+\epsilon)$-close to $F$, for some $t_0 < t$, and therefore $\text{bar}(G)_{t}\geq \text{bar}(F)_{t+\epsilon}$. As this happens for any $t$, we get that $\text{bar}(F)$ and $\text{bar}(G)$ are $\epsilon$-interleaved.
\end{proof}

Note that the minimal rank in the neighborhood $B(F,t)$ of a given functor $F$ can be obtained by non isomorphic functors, as can be seen in the following example.


\begin{example}\label{example mono}
Consider the compact and $1$-tame functor $F:\QQ^2\rightarrow \text{Vect}_K $ whose restriction to the sub-poset $\N^2\subset \QQ^2$ is described as follows. On the square $\{ v\leq (2,2)\}\subset \N^2$, $F$ is given by  the  commutative  diagram: 
 \[\xymatrix@R=12pt@C=12pt{
 K \rto^-{} & K \rto^-{} & K\\
 K\uto_-{}\rto^-{} & K\uto_-{} \rto^-{} & K \uto_{}\\
 0\rto\uto & K \uto_-{}\rto^-{} & K \uto_-{}
 }\]
 where the homomorphisms between non-zero entries are the identities.
For $w$ in $\N^2\setminus \{ v\leq (2,2)\}$, $F(\rm{meet}(w,(2,2))\leq w)$ is an isomorphism.
The feature counting invariant  associated to $F$, using the standard noise in the direction of the vector $(1,1)\in \QQ^2$, has values $\text{bar}(F)_t=2$, for $0\leq t \leq1$, and $\text{bar}(F)_t=1$ for every $t > 1$. The set $B(F,2)$ contains the following non isomorphic subfunctors of $F$ of rank one  $K((1,0),-),\, K((0,1),-)$ and $K((1,1),-)$. \end{example}

Computing the feature counting invariant of a functor $F\colon\QQ^r \rightarrow \text{\rm Vect}_K$ is in general not an easy task, since we do not have a formula or an algorithm which explicitly describes the sets $B(F,t)$ for $t$ in $\QQ$. One  strategy to calculate the value  $\text{\rm bar}(F)_{t}$ is to find proper subsets of 
the neighbourhood $B(F,t)$ where the minimal rank is achieved. Here is one such a subset.
Let $B'(F,t)$ be the collection of those tame functors $G$ for which there are natural transformations 
$F\leftarrow H:\!\phi$ and $\psi\colon H\to G$  such that $\phi$ is a $\tau$-equivalence and a {\em monomorphism},
$\psi $ is a $\mu$-equivalence, and $\tau+\mu< t$.  Then:

\begin{prop}\label{mono}
Let $F\colon\QQ^r\to \text{\rm Vect}_K$ be a tame and compact functor. Then:
\[\text{\rm bar}(F)_{t}=\text{\rm min} \{\text{\rm rank}(G) \,|\, G \in B(F,t)\}=
\text{\rm min} \{\text{\rm rank}(G) \,|\, G \in B'(F,t)\}\]
\end{prop}

\begin{proof}
Let $t$ be in $\QQ$ and $F\leftarrow H:\! \phi$ and $\psi\colon H\to G$ be natural transformations of tame functors such that $\phi$ is a $\tau$-equivalence, $\psi$ is a $\mu$-equivalence, $\tau+\mu<t$,  and $\text{bar}(F)_{t}=\text{rank}(G)$. Form the following push-out diagram:
\[\xymatrix@R=12pt@C=12pt{
& & H\ar@{->>}[dl]_{\phi} \drto^{\psi}\\
& \text{Im} \phi \drto_-{\psi^{\prime}}\ar@{^(->}[dl] & & G\ar@{->>}[dl]^-{\phi^{\prime}}\\
F && P
}\]
By Proposition \ref{prop pushpulleequiv} the morphism $\psi^{\prime}$ is a $\mu$-equivalence. 
Since $\phi$ is a $\tau$-equivalence, the same is true about the inclusion $\text{Im}\phi \hookrightarrow F$ and hence
$P$ belongs to $B(F,t)$.
 As $\phi^{\prime}$   is an epimorphism,  $\text{rank}(G)\geq \text{rank}(P)$. 
Therefore, by the minimality of the rank of $G$,  $\text{rank}(P)=\text{rank}(G)$.  
\end{proof}

\begin{cor}\label{zero rank}
Let $F\colon \QQ^r \rightarrow \text{\rm Vect}_K$ be a compact and tame functor and $t$ a positive rational number.  Then  $\text{\rm bar}(F)_{t}=0$ if and only if $F$ is contained in $\cS_{\epsilon}$ for some $\epsilon< t$.
\end{cor}

\begin{proof}
Let $0\colon \QQ^r \rightarrow \text{Vect}_K$ be the functor whose values are all zero.
If $F\colon \QQ^r \rightarrow \text{Vect}_K$ is in $\cS_{\epsilon}$ for $\epsilon< t$ then the morphism $0 \to F$ is an $\epsilon$-equivalence and therefore $0$ is contained in $B(F,t)$ and consequently $\text{\rm bar}(F)_{t}=0$.
On the other hand, if $\text{bar}(F)_{t}=0$, then  by Proposition~\ref{mono}, there is a monomorphism $F\hookleftarrow H:\!\phi$   such that $\phi$ is a $\tau$-equivalence, $H \to 0$ is a $\mu$-equivalence, and $\tau+\mu<t$.
By the additivity of noise systems, we can then conclude  $F$ is in $\cS_{\tau+\mu}$.
\end{proof}

Instead of trying to calculate the precise values of $\text{\rm bar}(F)$, one might try first to estimate them.
This can be done using the following propositions:
\begin{prop}\label{prop estimationsurj}
Let $F,F^\prime\colon \QQ^r \rightarrow \text{\rm Vect}_K$ be compact and tame functors.  If there is an epimorphism $\zeta \colon F\twoheadrightarrow  F^\prime$, then
$\text{\rm bar}(F)_t\geq \text{\rm bar}(F^\prime)_t$ for any $t$.
\end{prop}
\begin{proof}
Let $t$ be in $\QQ$. Let $F\leftarrow H:\! \phi$ and $\psi\colon H\to G$ be natural transformations of tame functors such that $\phi$ is a $\tau$-equivalence, $\psi$ is a $\mu$-equivalence, $\tau+\mu<t$,  and $\text{bar}(F)_{t}=\text{rank}(G)$. Form the following commutative diagram, where the square containing $\xi$ and $\psi$ is a push-out:
\[\xymatrix@R=15pt@C=15pt{
F\ar@{->>}[d]_-{\zeta} & H\lto_{\phi}\rto^{\psi}\ar@{->>}[d]   & G\ar@{->>}[d] \\
F^{\prime} & \text{Im}(\zeta\phi)\rto^-{\xi}\ar@{_(->}[l] & G^\prime
}\]
By Proposition \ref{prop pushpulleequiv} the morphism $\xi$ is a $\mu$-equivalence. 
Since $\zeta$ is an epimorphism, it induces an epimorphism between $\text{coker}(\phi)$ and
the quotient $F^{\prime}/\text{Im}(\zeta\phi)$. It then follows that the inclusion $\text{Im}(\zeta\phi)\subset F^{\prime}$ is a $\tau$-equivalence. This means that $G^\prime$ belongs to $B(F^\prime,t)$. As $G^\prime$ is a quotient of
$G$, then $\text{rank}(G)\geq \text{rank}(G^{\prime})$ proving  the inequality
$\text{\rm bar}(F)_t\geq \text{\rm bar}(F^\prime)_t$.
\end{proof}

\begin{cor}\label{cor directsum}
Let $F,F^\prime\colon \QQ^r \rightarrow \text{\rm Vect}_K$ be compact and tame functors. Then
$\text{\rm bar}(F\oplus F^\prime)_t\geq \text{\rm max}\{\text{\rm bar}(F)_t, \text{\rm bar}(F^\prime)_t\}$ for any
$t$ in $\QQ$. If the components of the noise system $\{\cS_{\epsilon}\}_{\epsilon \in \QQ}$ are closed under 
direct sums (see Section~\ref{noise}), then
$\text{\rm bar}(F)_t+\text{\rm bar}(F^\prime)_t\geq \text{\rm bar}(F\oplus F^\prime)_t$ for any
$t$ in $\QQ$.
\end{cor}
\begin{proof}
The first statement is a consequence of  Proposition~\ref{prop estimationsurj}. 
The second inequality, in the case the noise system is closed under direct sums, follows from the fact that
if $G$ and $G^{\prime}$ belong respectively to $B(F,t)$ and $B(F^\prime,t)$, then
$G\oplus G^\prime$ belongs to $B(F\oplus F^\prime,t)$.
\end{proof}

\section{The Feature counting Invariant for $r=1$}\label{one dimensional}
Let us choose a noise system $\{\cS_{\epsilon}\}_{\epsilon \in \QQ}$ in $\text{\rm Tame}(\QQ, \text{\rm Vect}_K)$.
Let $F\colon\QQ\to  \text{\rm Vect}_K$ be a tame and compact functor. 
 The main result of this section is Proposition~\ref{prop:bars1}, which holds for any noise system closed under direct sums. It states that  $\text{\rm bar}(F)_t$ counts the number of bars in the barcode (see~\cite{persistence}) of the quotient of $F$ by the maximal subfunctor that belongs to $\cS_{\epsilon}$. As in the classical case  the main tool in obtaining this result is the classification theorem of modules over a PID (see \ref{prop classification}). 
 
 Our first step is to show that to compute the value of $\text{\rm bar}(F)_{t}$ it is enough to only minimise the rank over  subfunctors of $F$. That is the reason we 
 define $B^{\prime\prime}(F,t)$ to be the collection of tame subfunctors $G\subset F$ for which this  inclusion is an $\epsilon$-equivalence,
 for some $\epsilon< t$ in $\QQ$.
\begin{prop}\label{barr=1}
Let $F\colon\QQ \to \text{\rm Vect}_K$ be a tame and compact functor. Then for any noise system in
$\text{\rm Tame}(\QQ, \text{\rm Vect}_K)$: 
\[\text{\rm bar}(F)_{t}=\text{\rm min} \{\text{\rm rank}(G) \,|\, G \in B(F,t)\}=
\text{\rm min} \{\text{\rm rank}(G) \,|\, G \in B^{\prime\prime} (F,t)\}\]
\end{prop}
\begin{proof}
The key property of $\text{\rm Tame}(\QQ, \text{\rm Vect}_K)$  we use is  that,  for any subfunctor $G\subset F$,
$\text{rank}(G)\leq \text{rank}(F)$. This is a consequence of the fact that   in  $\text{\rm Tame}(\QQ, \text{\rm Vect}_K)$   any subfunctor of a free functor   is also free.
We can use this to get natural transformations
 $F\hookleftarrow H:\!\phi$ and $\psi\colon H\twoheadrightarrow G$  such that $\phi$ is an inclusion and a $\tau$-equivalence,
$\psi$  is an epimorphism and a $\mu$-equivalence, $\tau+\mu<t$, and $\text{rank}(G)=\text{\rm bar}(F)_{t}=:n$.
Note that $\psi$ can be assumed to be an epimorphism by replacing $G$  with $\text{Im}(\psi)$ if necessary.
Let $\{g_1\in G(v_1),\ldots, g_n \in G(v_n)\}$ be a minimal set of generators for $G$ (see \ref{point minimal})
and $h_i$ be any element in $H(v_i)$ which is mapped via $\psi$ to $g_i$.  Since 
$H/\langle h_1,\ldots, h_n \rangle$ is a quotient of the kernel of $\psi$, the inclusion $\langle h_1,\ldots, h_n\rangle \subset H$ is a $\mu$-equivalence. It follows that $\langle h_1,\ldots, h_n\rangle\subset F$  is a $(\tau+\mu)$-equivalence.
As the rank of $\langle h_1,\ldots, h_n \rangle$ is $n$, the proposition follows.
\end{proof}

Recall that any compact and tame functor $F\colon\QQ\to \text{\rm Vect}_K$ is  isomorphic to a finite direct sum of the form $\bigoplus_{i\in I} [w_i,u_i)\oplus\bigoplus_{j\in J} K(v_j,-)$  (see~\ref{prop chartamer=1}). Furthermore
the isomorphism types of  these   summands are uniquely determined by the isomorphism type of $F$.
For a positive $t$ in $\QQ$, define:
\[I_t:=\{i\in I\ |\ [w_i,u_i)\not \in \cS_\epsilon\text{ for any }\epsilon<t\}\]
\[J_t:=\{j\in J\ |\ K(v_j,-)\not \in \cS_\epsilon\text{ for any }\epsilon<t\}\]
\begin{prop}\label{prop:bars1}
Assume  $\{\cS_{\epsilon}\}_{\epsilon \in \QQ}$ is a noise system in $\text{\rm Tame}(\QQ, \text{\rm Vect}_K)$
whose components are closed under direct sums.  Let $F\colon\QQ\to \text{\rm Vect}_K$  be a compact and tame functor.  Then  $\text{\rm bar}(F)_t=|I_t|+|J_t|$.
\end{prop}
\begin{proof}
Let:
\[F_1:=\bigoplus_{i\in I_t} [w_i,u_i)\oplus\bigoplus_{j\in J_t} K(v_j,-)\ \ \ \ \ \ F_2:=\bigoplus_{i\in I\setminus I_t} [w_i,u_i)\oplus\bigoplus_{j\in J\setminus J_t} K(v_j,-)\]
The functor $F$ is isomorphic to $F_1\oplus F_2$ and $F_2$ belongs to $\cS_\epsilon$ for some $\epsilon<t$.
We can then use Corollary~\ref{cor directsum} to conclude  $\text{bar}(F)_t=\text{bar}(F_1)_t$. 
Without loss of generality, we can therefore assume $F=F_1$, i.e., $I=I_t$ and $J=J_t$.

In that case one shows that for any surjection $\phi\colon F\twoheadrightarrow G$ where $G$ is in $\cS_{\epsilon}$ for some $\epsilon<t$, the kernel of $\phi$ has the same rank as $F$. The proposition then follows from~\ref{barr=1}.
\end{proof}

Assume  the components of the noise system $\{\cS_{\epsilon}\}_{\epsilon \in \QQ}$ 
in $\text{\rm Tame}(\QQ, \text{\rm Vect}_K)$ are closed under direct sums.
This implies that, for any $\epsilon$ in $\QQ$, there exists    the unique maximal  subfunctor $F[\cS_{\epsilon}]\subset F$  with respect to the property that
$F[\cS_{\epsilon}]$ is in $\cS_{\epsilon}$ (see~\ref{maximal noise}).
 Let $t>0$ be in $\QQ$. For any $\tau\leq \mu<t$, since $\cS_\tau\subset\cS_{\mu}$, we have  inclusions
$F[\cS_{\tau}]\subset F[\cS_{\mu}]\subset F$. Define $F[\cS_{<t}]:=\bigcup_{\tau<t} F[\cS_{\tau}]\subset F$.
\begin{cor}\label{cor r01quitoentden}
Let  $\{\cS_{\epsilon}\}_{\epsilon \in \QQ}$ be a noise system in $\text{\rm Tame}(\QQ, \text{\rm Vect}_K)$
whose components are closed under direct sums and  $F\colon\QQ\to \text{\rm Vect}_K$   a compact and tame functor.  Then  $\text{\rm bar}(F)_t=\text{\rm rank}\left(\text{\rm coker}(F[\cS_{<t}]\subset F)\right)$.
\end{cor}

\section{The Feature Counting Invatiant for the standard noise}\label{sec bbsn}
The strategy for computing the feature counting function can be further improved in the case of the standard noise
$\{V_{\epsilon}\}_{\epsilon \in \QQ}$ 
 in the direction of a cone  $V\subset \QQ^r$ (see~\ref{pt stannoisecone}).  Such a noise system is fixed throughout this section. As in 
 Section~\ref{one dimensional}, define $B^{\prime\prime}(F,t)$ to be the collection of tame subfunctors $G\subset F$ for which  this inclusion   is an $\epsilon$-equivalence
 for some $\epsilon< t$ in $\QQ$.

 \begin{prop}\label{subfunctor}
 Let $F\colon\QQ^r\to \text{\rm Vect}_K$ be a tame and compact functor. Then for the standard noise in the direction of a cone:
\[\text{\rm bar}(F)_{t}=\text{\rm min} \{\text{\rm rank}(G) \,|\, G \in B(F,t)\}=
\text{\rm min} \{\text{\rm rank}(G) \,|\, G \in B^{\prime\prime} (F,t)\}\]
\end{prop}

\begin{proof}
%
By Proposition \ref{mono} there are
natural transformations $F\hookleftarrow H:\!\phi$ and $\psi\colon H\to G$ such that
$\text{coker}(\phi)\in V_\tau$, $\text{coker}(\psi)\in V_{a}$, $\text{ker}(\psi)\in V_{b}$,  $\tau+a+b<t$, and
$\text{rank}(G)=\text{\rm bar}(F)_{t}=:n$.  Let $\{g_1\in G(v_1),\ldots, g_n \in G(v_n)\}$ be a minimal set of generators for $G$ (see \ref{point minimal}).
Since $\text{coker}(\psi)\in V_{a}$, there are vectors $w_1,\ldots, w_n$ in the cone $V$ such that
$||w_i||=a$ and the element  $g_i^{\prime}:=G(v_i<v_i+w_i)(g_i)$  is in the image of $\psi_{v_i+w_i}\colon H(v_i+w_i)\to G(v_i+w_i)$.   Let $h_i\in H(v_i+w_i)$ be any element that is mapped via $\psi_{v_i+w_i}$ to $g_i^{\prime}$.
Consider the subfunctor $F^{\prime}:=\langle \phi(h_1),\ldots, \phi(h_n)\rangle$ of $F$. We claim that the inclusion $F^{\prime}\subset F$  is $(\tau+a+b)$--equivalence, and hence $F^{\prime}$ belongs to $B(F,t)$. If the claim holds, since $\text{rank}(F^{\prime})\leq n$, by the minimality of the rank of $G$, we can conclude that
$\text{rank}(F^{\prime})= n$ and the proposition follows. 
The inclusion $F^{\prime}\subset F$  is the image of the composition:
\[\xymatrix{
\langle h_1,\ldots,h_n \rangle\ar@{^(->}[r] & H\ar@{^(->}[r]^{\phi} &F}\]
As $\phi$ is a $\tau$-equivalence, it is enough to show that $\langle h_1,\ldots,h_n\rangle \subset H$ is an
$(a+b)$-equivalence (see~\ref{prop additiveequiv}).
Set $H':=\langle h_1,\ldots,h_n\rangle \subset H$ and $G':=\langle g_1^{\prime},\ldots,g_n^{\prime}\rangle \subset G$
and consider the following commutative diagram:
\[\xymatrix@R=15pt@C=15pt{
 & \langle h_1,\ldots,h_n\rangle  \ar@{->>}[r]\ar@{^(->}[d] &  \langle g_1^{\prime},\ldots,g_n^{\prime}\rangle\ar@{^(->}[d]\\
 \text{ker}(\psi)\ar@{^(->}[r] \dto& H\rto^{\psi}\ar@{->>}[d] & G\ar@{->>}[d] \\
 \text{ker}(\bar{\psi})\ar@{^(->}[r] & H/H'\rto^{\bar{\psi}}& G/ G'
}\]
The square containing $\psi$ and $\bar{\psi}$
is a push-out square and therefore the natural transformation  $\text{ker}(\psi)\to  \text{ker}(\bar{\psi})$ is an epimorphism.
It follows that  $ \text{ker}(\bar{\psi})$ belongs to $V_b$. Furthermore by definition, $G/G'$ is in $V_a$ and hence so  is the image of $\bar{\psi}$. We can then use the additivity property of noise systems to conclude that
$H/H'$ belongs to $V_{a+b}$.
\end{proof}
According to~\ref{subfunctor}, to find the value $\text{bar}(F)_t$, we need to find the minimum rank of subfunctors of $G\subset F$ for which this inclusion is an $\epsilon$-equivalence for some $\epsilon<t$.
Consider the functor given in Example ~\ref{example mono} where $\text{bar}(F)_2=1$. Note that  a minimal set of generators for $F$ is given by any $g_1\not=0$ in $ F(1,0)$ and $g_2\not=0$ in $F(0,1)$.
In this case both  $\langle g_1\rangle $ and $\langle  g_2 \rangle $ are $2$-close to $F$. In this example,  to get the minimum rank among
functors in $B(F,2)$, we do not need to consider all subfunctors $G\subset F$, but only the subfunctors which are generated by subsets of a given  set of minimal generators for $F$.  This is also the case when restricting to functors with one-dimensional domain. However the corresponding statement fails in the general setting as illustrated by the following examples:

\begin{example} Assume  the characteristic of $K$ is not $2$.
Let us consider the standard noise in $\text{Tame}(\QQ^3,\text{Vect}_K)$ in the direction of the cone
$\text{Cone}(1,1,1)$.
Consider any  compact and $1$-tame functor $F\colon\QQ^3\rightarrow \text{\rm Vect}_K $ 
of rank $3$ whose restriction to the sub-poset
$\{ v\leq (1,1,1)\}\cup\{(1,1,2),(1,2,1),(2,1,1)\}\subset \N^3\subset \QQ^3$ is given by the following commutative  diagram:

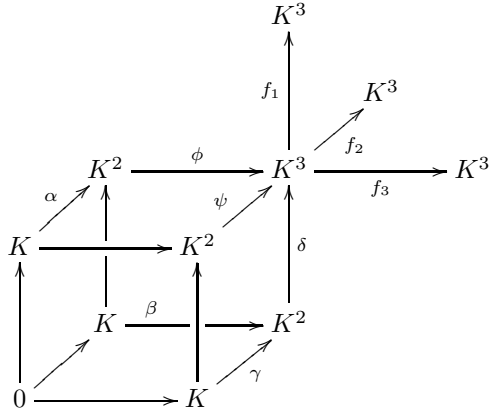
\[\xymatrix@C=15pt@R=15pt{
&&&K^3&\\
&&&&K^3\\
& K^2 \rrto^-{\phi} & & K^3 \uuto^-{f_1}\rrto_-{f_3}\urto_-{f_2} && K^3 \\
K \rrto_-{} \urto^-{\alpha} & & K^2\urto^(.4){\psi}\\
& K \uuto_-{}|\hole\rrto^(.25){\beta}|\hole & & K^2\uuto_-{\delta}\\
0\uuto\rrto\urto & & K \uuto_-{}\urto_-{\gamma}
}\]
where $\alpha=\gamma=\begin{pmatrix}
  1 \\
  0
   \end{pmatrix}$, \,$\beta=\begin{pmatrix}
  0 \\
  1
   \end{pmatrix}$, \,$\phi=\begin{pmatrix}
  1 & 0 \\
  0 & 1 \\
  0 & 0
   \end{pmatrix}$, \, $\delta=\begin{pmatrix}
  0 & 0 \\
  0 & 1 \\
  1 & 0
   \end{pmatrix},\, \psi=\begin{pmatrix}
  0 & 1 \\
  0 & 0 \\
  1 & 0
  \end{pmatrix},$
\[
f_1= \begin{pmatrix}
  1 & 1 & 0 \\
  1 & 0 & 1 \\
  0 & 0 & 0 
   \end{pmatrix}
,\  
f_2= \begin{pmatrix}
  1 & 1 & 0 \\
  0 & 1 & 1 \\
  0 & 0 & 0 
   \end{pmatrix}
,\ 
f_3= \begin{pmatrix}
  1 & 0 & 1 \\
  0 & 1 & 1 \\
  0 & 0 & 0 
   \end{pmatrix}.
\]
Let $g_1\in F(1,0,0)$, $g_2\in F(0,1,0)$ and $g_3\in F(0,0,1)$ be  non-zero vectors in  these  $1$-dimensional vector spaces.
These vectors form a minimal set of generators for $F$. Note that the vectors:
\begin{center}
$h_1:=F((1,0,0)\leq (1,1,1))(g_1)$\\ $h_2:=F((0,1,0)\leq (1,1,1))(g_2)$\\
$h_3:=F((0,0,1)\leq (1,1,1))(g_3)$\end{center}
form the standard basis for
$ F(1,1,1)=K^3$. The subfunctor $\langle h_1+h_2+h_3\rangle \subset F$ is $1$-close to $F$. This follows from 
the following equalities:
\begin{center}
 $f_1(h_1)=1/2 f_1(h_1+h_2+h_3)$\\ $f_2(h_2)=1/2 f_2(h_1+h_2+h_3)$ \\ $f_3(h_3)=1/2 f_3(h_1+h_2+h_3)$
 \end{center}
  Note further that  $f_1(h_3)$, $f_1(h_1)$ and $f_1(h_2)$ are pair-wise linearly independent. 
This implies that $\langle g_2 \rangle$ and $\langle g_3\rangle$ are not $1$-close to $F$. In the same way, $f_2(h_2)$ and $f_2(h_1)$ are not parallel and  therefore $\langle g_1\rangle$ is not $1$-close to $F$ either.   
\end{example}
\begin{example}
Let us consider the standard noise in $\text{Tame}(\QQ^2,\text{Vect}_K)$ in the direction of the cone
$\text{Cone}(1,1)$.
Consider any  compact and $1$-tame functor $F\colon\QQ^2\rightarrow \text{\rm Vect}_K $
of rank $2$ whose restriction to the sub-poset
$\{ v\leq (2,2)\}\subset \N^2\subset \QQ^2$ is given by the following commutative  diagram:

\[\xymatrix@C=40pt@R=40pt{
K\rto^{\text{id}} & K\rto & 0\\
K\rto_{\begin{pmatrix}0 \\ 1  \end{pmatrix}}\uto^{\text{id}}& K^2\rto^{\begin{pmatrix}1 & 0  \end{pmatrix}}
\uto^{\begin{pmatrix}1 & 1  \end{pmatrix}} & K\uto\\
0\uto\rto & K\uto_-{\begin{pmatrix}1 \\ -1  \end{pmatrix}}\rto^{\text{id}} & K\uto_{\text{id}}
}\]
Let $g_1\in F(1,0)$ and $g_2\in F(0,1)$ be non-zero vectors in these $1$-dimensional vector spaces.
These vectors form a minimal set of generators for $F$.  Note that neither $\langle g_1 \rangle\subset F$ nor $\langle g_2\rangle\subset F$  are $1$-close to $F$. However the subfunctor of $F$ generated by the element $(1,0)$  in  $K^2=F(1,1)$ is
$1$-close to $F$.
\end{example}

\section{denoising}\label{sec denoising}
The aim of this section is to introduce a notion of {\bf denoising} for tame and compact functors.  Intuitively, a denoising is an approximation and hopefully a simplification  of such a functor  that can be performed at different scales.  

\begin{Def}
Let  $\{\cS_{\epsilon}\}_{\epsilon\in \QQ}$ be a noise system in  $\text{Tame}(\QQ^r, \text{\rm Vect}_K)$ and 
 $F \colon\QQ^r \to \text{\rm Vect}_K$ be a tame and compact functor. A denoising of $F$ is a sequence of functors $\{\text{denoise}(F)_t\}_{0<t  \in \QQ}$, indexed by positive rational numbers, such that for  any $t $:
\begin{itemize}
\item $\text{denoise}(F)_t$ is in $B(F,t)$,
\item $\text{rank}(\text{denoise}(F)_t)= \text{bar}(F)_{t}$.
\end{itemize}
\end{Def}

Thus a denoising of $F$ at scale $t$ is a choice of a functor in the neighborhood $B(F,t)$ that realizes the minimum  value of the rank, which is given by $ \text{bar}(F)_{t}$. There are of course many such choices and there seems not to be a canonical one in general for $r>1$. 
Different denoising algorithms highlight different properties of the functor. Here we  present some examples of denoisings.

\begin{point}{\bf Minimal subfunctor denoising}. \label{minimal subfunctor}
Let $\{V_{\epsilon}\}_{\epsilon \in \QQ}$ be the standard noise 
 in the direction of a cone  $V\subset \QQ^r$, (see~\ref{pt stannoisecone}).  In this case
$ \text{bar}(F)_{t}$ can be obtained as the rank of some $G$ in  $B(F,t)$  which is a subfunctor of $F$ (see~ Proposition \ref{subfunctor}). 
Among these subfunctors with minimal rank we can then choose one which is also minimal with respect to the inclusion.
We call such a choice, a minimal subfunctor denoising of $F$ at $t$. For example:

\begin{example}\label{subfunctor denoising}
Consider the compact and $1$-tame functor $F\colon \QQ^2\rightarrow \text{\rm Vect}_K $  whose restriction to the sub-poset $\N^2\subset \QQ^2$ is described as follows. On the square $\{ v\leq (3,3)\}\subset \N^2$, $F$ is given by the    commutative  diagram: 
 \[\xymatrix@C=12pt@R=12pt{
 K \rto^-{} & K \rto^-{} & K \rto^-{}& K\\
 0 \rto^-{}\uto_-{} & K \rto^-{}\uto_-{} & K \rto^-{}\uto_-{}& K \uto_-{}\\
 0\uto_-{}\rto^-{} & 0\uto_-{} \rto^-{} & K \uto_{}\rto^-{}&K \uto_-{}\\
 0\rto\uto & 0 \uto_-{}\rto^-{} & 0 \uto_-{}\rto^-{}&K \uto_-{}
 }\]
 where the maps between non-zero entries are the identities.
 For $w$ in $\N^2\setminus \{ v\leq (3,3)\}$, the map $F(\text{meet}\{w,(3,3)\}\leq w)$ is an isomorphism.
 Consider the standard noise $\{V_{\epsilon}\}_{\epsilon \in \QQ}$  in the direction of  $\text{Cone}(1,1)$.
 For that noise system $\text{bar}(F)_2=2$.
 Let $\{ g_1 \in F(0,3),\, g_2 \in F(2,1)\, , g_3 \in F(1,2),\,  g_4 \in F(3,0)\}$ be a minimal set of generators for $F$.
 %
Set $g_1^{\prime}:=F((0,3)\leq(1,3))(g_1)$ and $g_4^{\prime}:=F((3,0)\leq(3,1))(g_4)$. The subfunctor $<g_1^{\prime}, g_4^{\prime}> \subset F$ is in $B(F,2)$, has the required  rank $2$ and is the minimal element, with respect to inclusion, among subfunctors of $F$ with  rank $2$ in $B(F,2)$.  
 
 \end{example}

In general the minimal subfunctor denoising is not unique: 
\begin{example}
Consider the compact and $1$-tame functor $F\colon \QQ^2\rightarrow \text{\rm Vect}_K $ 
 whose restriction to the sub-poset $\N^2\subset \QQ^2$ is described as follows. 
For $\{ v\leq (3,5)\}\subset \N^2$, $F$ is given by the    commutative  diagram: 
\[\xymatrix@C=12pt@R=12pt{
 K \rto^-{} & K \rto^-{} & K \rto^-{}& K \rto^-{}&K \rto^-{} & K   \\
 0 \rto^-{}\uto_-{} & K \rto^-{}\uto_-{} & K \rto^-{}\uto_-{}& K \uto_-{}  \rto^-{}&K \uto_-{}  \rto^-{} & K \uto_-{}\\
 0\uto_-{}\rto^-{} & 0\uto_-{} \rto^-{} & 0 \uto_{}\rto^-{}&K \uto_-{}  \rto^-{}&K \uto_-{}  \rto^-{} & K \uto_-{}\\
 0\rto\uto & 0 \uto_-{}\rto^-{} & 0 \uto_-{}\rto^-{}&0 \uto_-{}  \rto^-{} & 0 \uto_-{}  \rto^-{} & K \uto_-{}
 }\]
 where the maps between non-zero entries are the identities.
For $w$ in $\N^2\setminus \{ v\leq (3,5)\}$, the map $F(\text{meet}\{w,(3,5)\}\leq w)$ is an isomorphism.
 Consider the standard noise $\{V_{\epsilon}\}_{\epsilon \in \QQ}$  in the direction of  $\text{Cone}(1,1)$.
 For that noise $\text{bar}(F)_2=2$.
Let $\{ g_1 \in F(0,3),\, g_2 \in F(1,2)\, , g_3 \in F(3,1),\,  g_4 \in F(5,0)\}$ be a minimal set of generators for $F$. 
Let $g_1^\prime=F\left((0,3)<(1,3)\right)(g_1)$ and $g_4^\prime=F((5,0)<(5,1))(g_4)$.
The subfunctors of $F$ given by $<g_1^\prime,g_3>$ and  $<g_2,g_4^\prime>$ are in $B(F,2)$ and they have required rank $2$. Thus they are examples of $2$-denoising of $F$. Furthermore they are both minimal with respect to inclusion.
 \end{example}

\end{point}

\begin{point}{\bf Quotient denoising}. \label{quotient denoising}
Let $\{\cS_{\epsilon}\}_{\epsilon\in \QQ}$ be a noise system in $\text{\rm Tame}(\QQ^r,\text{\rm Vect}_K)$ whose compact part is closed under direct sums
(see~\ref{prop directsumstandnoisecone} and~\ref{sequence of vectors}  for  when a   standard noise in the direction of a cone $V\subset \QQ^r$ or a sequence of vectors $\V=\{v_1,\ldots, v_n\}$ in $\QQ^r$ is closed under direct sums). 
 For any $F\colon\QQ^r\to\text{\rm Vect}_K$ there exists a unique maximal noise of size $\epsilon$ contained in $F$,  $F[\cS_{\epsilon}]\subset F$  (see~\ref{maximal noise}). Given a tame and compact functor 
$F \colon\QQ^r \to \text{\rm Vect}_K$, define the subfunctor $F[\cS_{<t}]:=\bigcup_{\tau<t} F[\cS_{\tau}]\subset F$.
 One can ask if the sequence $\{\text{coker}(F[\cS_{<t}]\subset F)\}_{0<t \in \QQ}$ is a denoising of $F$.  If it is a denoising, we call it the quotient denoising of $F$.
Corollary~\ref{cor r01quitoentden} states that in the case $r=1$, this procedure always gives a denoising of $F$. 
\end{point}
\begin{example}
Consider the compact and $1$-tame functor $F\colon \QQ\rightarrow \text{\rm Vect}_K $ whose restriction to the sub-poset $\N\subset \QQ$  is described as follows. On the segment $\{ v\leq 4\}\subset \N$, $F$   is    given by  the commutative diagram:
\[\xymatrix@C=35pt@R=30pt{
 K^3 \rto^-{{\tiny
 \begin{pmatrix}
 1&0&1\\
 1&1&1
 \end{pmatrix}
 }} & K^2 \rto^-{{\tiny
 \begin{pmatrix}
 1&0\\
 0&0
 \end{pmatrix}
 }} & K^2 \rto^-{{\tiny
 \begin{pmatrix}
 1&1
 \end{pmatrix}
 }} & K \rto^-{1} & K  
 }\]
For any $w \geq 4$, $F(4\leq w)$ is an isomorphism.
Let $\{V_{\epsilon}\}_{\epsilon \in \QQ}$ be standard noise in the direction of the full cone $\QQ$. The basic barcode of $F$ has values:

\[\text{bar}(F)_{t}=\begin{cases} 4 & \text{ if } 0 \leq t \leq 1\\
2 & \text{ if } 1 < t \leq 2\\
1 & \text{ if } t > 2.
\end{cases} 
\]

The cokernel of the inclusion $F[S_{<2}]\subset F$ is isomorphic to:
\[\xymatrix@C=35pt@R=30pt{
 K^2 \rto^-{{\tiny
 \begin{pmatrix}
 1&0
 \end{pmatrix}
 }} & K \rto^-{1} & K \rto^-{1} & K \rto^-{1} & K  
 }\]
For $t > 2$, the cokernel of the inclusion $F[S_{< t}]\subset F$ is isomorphic to:
\[\xymatrix@C=35pt@R=30pt{
 K \rto^-{1} & K \rto^-{1} & K \rto^-{1} & K \rto^-{1} & K  
 }\]
\end{example}
Note that for $r>1$, the family of functors $\{ \text{coker}(F[\cS_{<t}]\subset F)\}_{0<t \in \QQ}$ is not always a denoising.
\begin{example}
Let $F\colon\QQ^2\to\text{\rm Vect}_K$ be the functor defined in Example \ref{example mono} and $\{V_{\epsilon}\}_{\epsilon\in \QQ}$ is standard noise in the direction of  $\text{Cone}(1,1)$. 
Since $F(v\leq w)$ is a monomorphism for any $v\leq w$ in $\QQ^2$, the quotient denoising of $F$ at scale $t$ is isomorphic to $F$, for any $t$ in $\QQ$.
It follows that the rank of $\text{coker}(F[\cS_{< 2}]\subset F)$ is $2$, while $\text{bar}(F)_2=1$.
\end{example}

 Given a denoising $\{\text{denoise}(F)_t\}_{0<t  \in \QQ}$, we can consider the family of multisets $\{\beta_0 \text{denoise}(F)_t\}_{0<t  \in \QQ}$. Note that in the case of quotient denoising such family of multisets has the property that if $s<t$ in $\QQ$ then $\beta_0 \text{denoise}(F)_t$ is a subset of $\beta_0 \text{denoise}(F)_s$, (see \ref{pt multiset}). 
It will be the focus of future work to study the stability of families of multisets associated to a denoising and how such invariants identify persistent features (see \cite{Multi}).

\section{Future Directions}
\begin{point}
The feature counting invariant   is   convenient in estimating the number of 
significant features of a mutidimensional persistence module relative to a given noise system.
However for identifying these features an appropriate denoising scheme (see Section~\ref{sec denoising})  is
crucial.  Existence of such a denoising follows for a given noise system whenever the following is true: 
\smallskip

\noindent
{\em  for  $G_1$ and $G_2$ in $B(F,\epsilon)$  that have the minimal rank, given necessarily  by  $\text{bar}(F)_{\epsilon}$, the
 Betti diagrams   $\beta_0G_1,\beta_0G_2\colon \QQ^r\to \N$ are $2\epsilon$-interleaved.}
 \smallskip

As of the writing of this paper we do not know for what noise system the above statement holds true. We believe however that characterising such systems is a worthwhile pursuit and we hope to return to it in the near future. A related and possible easier  question is if there exist some natural class of multidimensional persistence modules for which the above statement holds true for the standard noise. 
\end{point}

\begin{point}
A second possible avenue for future work is the construction and subsequent implementation of algorithms for computing feature counting invariants. 
An implementation of the feature counting invariant will help us in understanding which type of noise best detects persistent features given a specific application or construction.
As the computation of feature counting invariants  involves rank minimisation one would expect that this is an NP-hard problem. However for applications to data analysis  determining exactly the feature counting invariant is not necessarily the ultimate aim. Indeed it is rather the order of magnitude of the values of this  invariant that is important.  Intuitively one may think of feature counting invariants as a measure of the complexity of a space and showing that one space is ever so slightly more complex than another is probably of neglectable interest for applications. On the other hand  a big   difference between  the values of the  feature counting invariants for a given $\epsilon$  (for example $1$, $10$ or $10 000$)  has more drastic implications  on the underlying geometrical structures of the spaces at hand. Therefore the computability of the more feasible question of bounding or approximating the feature counting function is more interesting. 
\end{point} 

\begin{point}{\bf Serre localization.}
Let $\{\cS_{\epsilon}\}_{\epsilon\in \QQ}$ be a noise system. 
Recall that the $0$-th component $\cS_0$ is always a Serre subcategory in
$\text{\rm Tame}(\QQ^r,\text{\rm Vect}_K)$.
Similarly, for any $t$ in $\QQ$, so are  the unions
$\cup_{\epsilon\geq t}\cS_\epsilon$ and
$\cup_{\epsilon> t}\cS_\epsilon$.  Being Serre means that we can quotient
out these  subcategories and  localize  $\text{\rm Tame}(\QQ^r,\text{\rm Vect}_K)$
away from them.  This process is well explained in~\cite{Serre}.  
Since $\cup_{\epsilon> t}\cS_\epsilon\subset \cup_{\epsilon\geq t}\cS_\epsilon$, we get 
the following commutative diagram of functors where the vertical maps denote the appropriate localizations and 
$\phi$ is given by the universal property: 
\[\xymatrix{
& \text{\rm Tame}(\QQ^r,\text{\rm Vect}_K)\dlto_{L_{>t}}\drto^{L_{\geq t}}\\
\overline{\text{\rm Tame}(\QQ^r,\text{\rm Vect}_K)}\rrto^{\phi} & & \overline{\overline{\text{\rm Tame}(\QQ^r,\text{\rm Vect}_K)}}
}\]
We believe that understanding the relation between denoising at scale $t$ and the functor $\phi$ is a problem worth pursuing. As of the writing of this paper we were unable to use these Serre localizations to give what we could call a ``better'' conceptual explanation for the feature counting invariants or to construct new continuous invariants.
\end{point}

\section{Appendix:  functors indexed by $\N^{r}$}\label{appendixfn}
The aim of this appendix is to recall   basic   properties of the category of functors $\text{Fun}(\N^{r},\text{Vect}_K)$: we  identify its compact objects,  projective objects, and discuss minimal covers. 
Although it is standard, we decided to include this material  for self containment.


\begin{point}\label{pt semisimpl}
{\bf Semisimplicity.}
A functor $F\colon \N^{r}\to \text{Vect}_K$ is  {\bf semi-simple} if $F(v<w)$ is the zero homomorphism for any $v<w$.
For example the unique functor  $U_v\colon \N^{r}\to \text{Vect}_K$ such that $U_v(v)=K$ and $U_v(w)=0$ if $w\not =v$
is semi-simple.  Any semi-simple functor is a direct sum of functors of the form $U_v$ and hence can be described
uniquely as   $\oplus_{v\in \N^r} (U_v\otimes V_v)$ for some sequence of
vector spaces $V_v$.
\end{point}

\begin{point}\label{pt radical}
{\bf Radical.}
Let $F\colon \N^{r}\to \text{Vect}_K$ be a functor. Define $\text{rad}(F)(v)$ to be the subspace of $F(v)$ given by the sum
of all the images of $F(u< v)\colon F(u)\to F(v)$ for all $u< v$. For any $v\leq w$,
the homomorphism $F(v\leq w)\colon F(v)\to F(w)$ maps the subspace $\text{rad}(F)(v)\subset F(v)$ into 
  $\text{rad}(F)(w)\subset F(w)$. Thus  these subspaces form a subfunctor denoted by
$\text{rad}(F)\subset F$.   A natural  transformation $\phi\colon F\to G$ 
maps the subfunctor $\text{rad}(F)\subset F$ into the subfunctor $\text{rad}(G)\subset G$. The resulting
natural transformation is denoted by $\text{rad}(\phi)\colon \text{rad}(F)\to \text{rad}(G)$.
Note that for any functor $F\colon \N^{r}\to \text{Vect}_K$, $F/\text{rad}(F)$ is semisimple. 


\begin{prop}\label{prop epi}
A natural transformation $\phi\colon F\to G$ in $\text{\rm Fun}(\N^{r},\text{\rm Vect}_K)$  is an epimorphism if and only if its composition with the quotient
$\pi\colon G\to G/\text{\rm rad}(G)$ is an epimorphism.
\end{prop}
\begin{proof}
Since $\pi$ is an epimorphism, if $\phi$ is an epimorphism, then so is $\pi\phi$. Assume that  $\pi\phi$ is an epimorphism.
If $\phi$ is not an epimorphism, then the set of   $v$ in $\N^r$ for which $\phi(v)\colon F(v)\to G(v)$
is not an epimorphism is not empty.  Let us choose a minimal element $u$ in this set (see~\ref{point NRposets})
and consider a commutative diagram with exact rows:
\[\xymatrix@R=15pt{
0\rto & \text{rad}(F)(u) \dto_{\text{rad}(\phi)(u)}\rto & F(u)\rto\dto^{\phi(u)} & F(u)/\text{rad}(F)(u)\dto \rto & 0\\
0\rto &  \text{rad}(G)(u) \rto & G(u)\rto^-{\pi} & G(u)/\text{rad}(G)(u) \rto & 0
}\]
Minimality of $u$ implies
that $\text{rad}(\phi)(u)$ is an epimorphism.  Since  the  right side   vertical homomorphism is also  an epimorphism, it follows that so is the middle one,
contradicting the assumption about $\phi(u)$. 
\end{proof}
Applying \ref{prop epi} to $0\hookrightarrow F$, we get that  $\text{rad}(F)=F$ if and only if $F=0$. The very same argument as in the proof of~\ref{prop epi} can also be used to show:

\begin{prop}\label{prop fdimvalues}
A functor $G\colon\N^r\to \text{\rm Vect}_K$ has finite dimensional values   if and only if $G/\text{\rm rad}(G)$
has 
finite dimensional values.
\end{prop}
\end{point}

\begin{point}
{\bf Minimal covers.}  
Recall that   $\phi\colon F\to  G$ in  $\text{\rm Fun}(\N^{r},\text{\rm Vect}_K)$  is called  minimal  if
any $f\colon F\to F$, such that $\phi=\phi f$, is an isomorphism (see~\ref{point minimal}). For example:

\begin{prop}\label{prop exminepi}
Let   $\phi\colon F\to G$ in  $\text{\rm Fun}(\N^{r},\text{\rm Vect}_K)$ be a natural transformation such that 
the induced morphism on the quotients  $[\phi]\colon F/\text{\rm rad}(F)\to G/\text{\rm rad}(G)$  is an isomorphism and either the values of $F$ are finite dimensional or $F$ is free. 
Then  $\phi$ is  minimal.
 \end{prop}
\begin{proof}
Let $f\colon F\to F$ be such that $\phi=\phi f$.  By quotienting  out radicals we obtained a commutative triangle:
\[\xymatrix@R=13pt{
 F/\text{\rm rad}(F)\drto_{[\phi]}  \rrto^-{[f]}& & F/\text{\rm rad}(F)\dlto^{[\phi]} \\
 & G/\text{\rm rad}(G)
}\]
  Since $[\phi]$ is an isomorphism, then so is
$[f]$. We can then use~\ref{prop epi} to conclude that $f\colon F\to F$ is an epimorphism.   As epimorphisms of finite dimensional vector spaces are isomorphisms, under the assumption that $F$ has finite dimensional values, we can conclude that  $f$ is an isomorphism.  In the case  $F$  is free, the map $f$ splits and $F$  is isomorphic to $F\oplus \text{ker}(f)$. It  follows that  $\text{ker}(f)/\text{rad}(\text{ker}(f))=0$ and hence $\text{ker}(f)=0$ showing that in this case $f$ is also an isomorphism. 
\end{proof}
\begin{point}
Let $V$ be a vector space. Consider the functor $K(v,-)\otimes V\colon \N^r\to \text{Vect}_K$. The Yoneda isomorphism
(see~\cite{Categories} or ~\cite{HomologicalAlgebra})
states that,  the function   assigning to a natural
transformation  
$\phi\colon K(v,-)\otimes V\to G$ the homomorphism $V\to G(i)$ given by  $x\mapsto \phi(v)(\text{id}_v\otimes x)$
 is an isomorphism between 
 $\text{Nat}(K(v,-)\otimes V,G)$ and $\text{Hom}(V,G(v))$.
A direct consequence of this isomorphism is the fact that $G\mapsto \text{Nat}(K(v,-)\otimes V,G)$ is an exact  operation (i.e.,  $K(v,-)\otimes V$ is a projective object in $\text{\rm Fun}(\N^r,\text{\rm Vect}_K)$, see~\cite[Section 2.2]{HomologicalAlgebra}).  

\end{point}

\begin{point}\label{pt invsemisimple}
Consider a functor $G\colon \N^{r}\to \text{\rm Vect}_K$. Since $G/\text{rad}(G)$ is semisimple,
it is isomorphic to  $\oplus_{v\in \N^r} (U_v\otimes V_v)$ for some sequence of
vector spaces $V_v$.
Let $F=\oplus_{v\in \N^r} K(v,-)\otimes V_v$. Note that $F/\text{rad}(F)$ is also isomorphic to
$\oplus_{v\in \N^r} (U_v\otimes V_v)$. Since $F$ is projective (see~\ref{point freefunctors}), then there is a natural transformation
$\phi\colon F\to G$ making the following diagram commutative:
\[\xymatrix@R=13pt{
F\rto^{\phi}\dto_-{\pi} & G\dto^-{\pi}\\
F/\text{rad}(F)\rto^{\simeq} & G/\text{rad}(G)
}\]  
where $\pi$'s are the quotient transformations.
According to~\ref{prop exminepi} the natural transformation $\phi\colon F\to G$ is  minimal. Since it is also an epimorphism
(see~\ref{prop epi}) and $F$ is free, this map is a minimal cover (see~\ref{point minimal}). Thus all functors in $\text{\rm Fun}(\N^{r},\text{\rm Vect}_K)$ admit a minimal cover.
Moreover  $G$  is of finite type (see~\ref{point minimal}) if and only if
$V_v$ is finite dimensional for any $v$ in $\N^r$. Its support (see~\ref{point minimal}) is given by the subset of $\N^r$ of all elements $v$ for which $V_v\not= 0$. If $G$ is of finite rank (see~\ref{point minimal}), then its rank is given by $\sum_{v\in \N^r}\text{\rm dim}_KV_v$.
If   $G$ÃÂ   is of finite type, then the multiset $\N^r\ni v\mapsto \text{\rm dim}_KV_v\in \N$ is the $0$-Betti diagram of $G$ (see~\ref{point minimal}).   Moreover $F$ has a finite set of generators if and only if it is of finite rank.
Note that being of finite type, of finite rank, and the invariants $\text{supp}(G)$, $\text{rank}(G)$ and  $\beta_0G$ depend only on $G/\text{rad}(G)$.  However, the choice of a set of minimal generators for $G$ is equivalent to a choice of a minimal cover $F\to G$ and hence it   contains much more information, than the semisimple functor $G/\text{rad(G)}$. 
 \end{point}

\begin{cor}\label{cor minpresentation}
A natural transformation   $\psi\colon H\to G$ in  $\text{\rm Fun}(\N^{r},\text{\rm Vect}_K)$  is a minimal cover if and only if $H$  is free and the induced morphism on the quotients  $[\psi]\colon H/\text{\rm rad}(H)\to G/\text{\rm rad}(G)$  is an isomorphism.
\end{cor}
\begin{proof}
Assume $\psi\colon H\to G$ is a minimal cover. Let $V_v$ be a sequence of vector spaces such that
$G/\text{\rm rad}(G)$ is isomorphic to $\oplus_{v\in \N^r} (U_v\otimes V_v)$.  Consider $F=\oplus_{v\in \N^r} K(v,-)\otimes V_v$ and a minimal cover $\phi \colon F\to G$ described above. All these natural transformations fit into the following commutative diagram:
\[\xymatrix@R=15pt{
H\rto^(.7){\psi} \ar@/^15pt/[rr]|{g} \dto^-{\pi} &G\dto^-{\pi} & F\lto_(.7){\phi}\dto^-{\pi}\\
H/\text{rad}(H)\rto^{[\psi]}  \ar@/_15pt/[rr]|{[g]}& G/\text{rad}(G) &F/\text{rad}(F)\lto_{[\phi]}
}\]
Note that $[\phi]$ is an isomorphism  by construction and  $g$ is an isomorphism by the minimality
assumption on $\psi$ and $\phi$. It then follows that $[g]$ is an isomorphism and consequently so is $\psi$.
That proves one implication of the corollary. The other implication follows from~\ref{prop exminepi} and~\ref{prop epi}. 
\end{proof}

\begin{cor}
Any projective object in  $\text{\rm Fun}(\N^{r},\text{\rm Vect}_K)$ is free.
  \end{cor}
  \begin{proof}
  Let $P\colon \N^{r}\to\text{\rm Vect}_K$ be projective and $\phi\colon F\to P$ be a minimal cover.
  As $P$ is projective, there is a natural transformation $s\colon P\to F$ such that $\phi s=\text{id}_P$.
  This implies that $s$ is a monomorphism.
  According to~\ref{prop epi} and \ref{cor minpresentation} $s$ is also an epimorphism. We can conclude that $s$ is an isomorphism and hence
  $P$  is free.  
 \end{proof}
 \end{point}

\begin{point}\label{pt compact}
{\bf Compact objects.}  Recall that an object $A$ in an abelian category is compact if, for any sequence
of monomorphisms $A_1\subset A_2\subset\cdots \subset A$  such that 
$A= \text{colim}A_i$, there is $k$ for which  $A_k=A$ (see \cite{compact}).

\begin{prop}\label{prop comactfunctors}
Let  $F\colon \N^{r}\to  \text{\rm Vect}_K$ be a functor. The following are equivalent:
\begin{enumerate}
\item[(a)] $F$ is compact in  $\text{\rm Fun}(\N^{r},\text{\rm Vect}_K)$;
\item[(b)] $F$ is of finite rank;
\item[(c)] $F$ fulfills the following two conditions:
(1) $F(v)$ is finite dimensional for any $v$;   (2) there is a natural number
$n$  such that, for any $v=(v_1,\ldots v_r)$ in $\N^{r}$,  the homomorphism
 $F((\text{\rm min}(n,v_1),\ldots, 
\text{\rm min}(n,v_r))\leq v)$ is an isomorphism.
\end{enumerate}
\end{prop}
\begin{proof}
Assume  $F$ is compact. If $F$ is not of finite rank, then there is a sequence
of proper subfunctors $G_1\subset G_2\subset \cdots \subset F/\text{rad}(F)$ such that 
$\text{colim}G_i=F/\text{rad}(F)$. Let $F_i:=\pi^{-1}(G_i)$ where $\pi\colon F\to F/\text{rad}(F)$  is the quotient
transformation. We obtain   a  sequence of proper subfunctors $F_1\subset F_2\subset\cdots\subset F$ such that $\text{colim}F_i=F$ contradicting the fact that $F$ is compact. This proves the implication 
(a)$\Rightarrow$(b).

Assume   $F$ is of finite rank.  The first condition in (c) follows from~\ref{prop fdimvalues}.
Let $i$ be a natural number. For   $v=(v_1,\ldots,v_r)$  in $\N^{r}$, set
$v^{i}:=(\text{min}(i,v_1),\ldots, 
\text{min}(i,v_r))$.  Note  if $v\leq w$, then $v^{i}\leq w^{i}$. Define $F^i(v):=F(v^{i})$ and
$F^i(v\leq w):=F(v^{i}\leq w^{i})$. In this way we obtain a functor  $F^i\colon \N^{r}\to  \text{\rm Vect}_K$.
If $i\leq j$, then $v^{i}\leq v^{j}\leq v$. Let $F^{i\leq j}\colon F^i\to F^j$ be the natural transformation given by the homomorphisms 
$F(v^{i}\leq v^{j})$ and  $F^{i< \infty}\colon F^i\to F$  be the natural transformation
given by  the homomorphisms $F(v^{i}\leq v)$. Since
$\text{colim}_{i}(F^{i<\infty})\colon \text{colim} (F^i)\to F$ is an isomorphism and  $F$ is of finite rank, there is
$m$ such that  $F^{m< \infty}\colon F^m\to F$ is surjective. From the definition it follows that 
$F^{m\leq i}\colon F^m\to F_i$ is also surjective.  Let $T:=\{v\in \N^{r}\ |\ v=v^m\}$, $v\in T$, and $i\geq m$.
Note that $T$ is a  finite set of all elements $v$ in $\N^{r}$ such that $v\leq (m,\ldots,m)$. 
 Define:
 \[K_{i,v}=\cup_{v\leq w}\text{ker}\left(F^m(v)=F(v)\xrightarrow{F(v\leq w^{i})} F(w^{i})=F^i(w)\right)\] 
 \[K_i:=\bigoplus_{v\in T} K_{i,v}\subset \bigoplus_{v\in T} F(v)\]
Note that $K_i\subset K_j$ if $i\leq j$. Since $T$ is  finite, the space $\oplus_{v\in T} F(v)$ is finite dimensional, and hence 
there is $n$ such that $K_n=K_i$ for any $i\geq n$. For this $n$, the natural transformation  $F^{n\leq i}\colon F^n\to F^i$   is an isomorphism. It follows that so is $F^{n<\infty}\colon F^n\to F$ proving the second condition
in (c).

 Assume  (c). Let $n$ be the number given by second condition in (c). To prove  $F$ is compact consider 
   a sequence $F_1\subset F_2\subset\cdots \subset F$  of  sub-functors such that
   $F=\text{colim}F_i$. Since $F$ has  finite dimensional
 values, there is $m$ such that, for any $w=(w_1,\ldots, w_r)$ with $w=w^n$,
 $F_m(w)=F(w)$.  This together with condition (2) in (c) implies  
 that, for any $v$, $F_m(v)=F(v)$.  The functor $F$ is therefore compact proving the implication  (c)$\Rightarrow$(a).
 \end{proof}

As already mentioned in Section $3$, the category $\text{Fun}({\N}^r,\text{Vect}_K)$ is equivalent to the category of $r$-graded modules over the polynomial ring in $r$-variables with the standard $r$-grading.
In this context (see \cite{gradmodules}) Proposition \ref{prop comactfunctors} can be rephrased as:
\begin{prop}\label{prop modules}
Let  $F$ be an $\N^{r}$-graded module over the polynomial ring $K[x_1,\ldots,x_r]$. The following are equivalent:
\begin{enumerate}
\item[(a)] $F$ is Notherian;
\item[(b)] $F$ is finitely generated;
\item[(c)] $F$ is positively $(n,\ldots ,n)$-determined for some $n \in \N$.
\end{enumerate}
\end{prop}
While the equivalence between $(a)$ and $(b)$ depends on the fact that $K[x_1,\ldots,x_r]$ is a Notherian ring, the equivalence of $(c)$ with $(b)$ is proved in Proposition $2.5$ of \cite{gradmodules}.

A direct consequence of~\ref{prop comactfunctors} is that  all the  quotients and all the subfunctors of a compact functor in $\text{\rm Fun}(\N^{r},\text{\rm Vect}_K)$ 
are compact. 
\end{point}

We finish this section with recalling (see~\cite{persistence}) the classification of compact objects in $\text{\rm Fun}(\N,\text{\rm Vect}_K)$:
\begin{point}\label{pt compactr=1}
{\bf Compact objects in $\text{\rm Fun}(\N,\text{\rm Vect}_K)$.}  
Let $w\leq u$ be in $\N$. There is a unique inclusion $K(u,-)\subset K(w,-)$. The cokernel of this inclusion is denoted by  $[w,u)$ and called the bar starting in $w$ and ending in $u$.  Note that such functors are compact of rank $1$ whose $0$-Betti diagram is given by:
\[\beta_0[u,w)(v)=\begin{cases} 1& \text{ if } v=u\\
0&  \text{ if } v\not =u
\end{cases}\]
\begin{prop}\label{prop classification}
Any compact object in $\text{\rm Fun}(\N,\text{\rm Vect}_K)$ is isomorphic to a finite direct sum of functors of the form $[w,u)$ and $K(v,-)$. 
Moreover the isomorphism types of these   summands are uniquely determined by the isomorphism type of the functor.
\end{prop}
\begin{proof}
Assume  not all the compact functors can be expressed as  a direct sum of bars and free functors and let $G$ be a such a functor with  minimal rank.  Since $G$   is compact, there is $l$ in $\N$, such that $G(l\leq v)$ is an isomorphism for any $v\geq l$. Let $w:=\text{min}\{v\in \N\ |\ G(v)\not=0\}$. Such $w$ exists since $G$ can not be the zero functor. Note that $w\leq l$. Choose an element $x\not=0$ in $G(w)$.   Consider the set $\{v\in \N\ |\ v\geq w\text{ and } G(w\leq v)(x)\not=0\}$. If this set
does not have a maximum, define $G(l)\to K$ to be any map that maps the element $G(w\leq l)(x)$ to $1$.
This linear map can be extended uniquely to a surjective map $\phi\colon G\to K(w,-)$. Since  $K(w,-)$ is projective,
$G$ is a direct sum of $K(w,-)$ and $\text{ker}(\phi)$. As $\text{ker}(\phi)$ has a smaller rank than $G$, it can be expressed as a direct sum of bars and free functors. It would then follow  $G$ itself is such a direct sum, contradicting the assumption. We can then define $u=\text{max}\{v\in \N\ |\ v\geq w\text{ and } G(w\leq v)(x)\not=0\}$ and set  $G(u)\to K$ to be any map that maps the element $G(w\leq u)(x)$ to $1$. This linear map can be extended uniquely to a surjective map $\phi\colon G\to [w,u)$. This map has  a section given by the inclusion $[w,u)\subset G$
which maps $1$ in $[w,u)(w)=K$ to $x$. The functor $G$ can be then expressed as a direct sum $[w,u)\oplus \text{ker}(\phi)$. That  leads to a contradiction by the same argument as before. That means that such a $G$
does not exists and all compact functors can be expressed as direct sums of bars and free functors.

For the uniqueness, note that  if $G$ is isomorphic to $ \bigoplus [w_i,u_i)\oplus\bigoplus K(v_j,-)$, then
$\beta_0G$
determines the starting points $w_i$'s and $v_j$'s and hence these numbers are uniquely determined
by $G$.  Let us choose a minimal  free cover $F\to G$.  The ends $u_i$'s  are determined by
$\beta_0\text{ker}(F\to G)$ and hence again they depend only on the   isomorphism type of $G$.
\end{proof}
\end{point}

\end{document}